\documentclass{amsart}
\usepackage[latin9]{inputenc}
\usepackage{geometry}
\usepackage{verbatim}
\usepackage{amstext}
\usepackage{amsthm}
\usepackage{amssymb}
\usepackage{graphicx}
\usepackage{subcaption}
\usepackage[hidelinks]{hyperref}
\usepackage[french,english]{babel}

\makeatletter
\numberwithin{equation}{section}
\numberwithin{figure}{section}

\theoremstyle{plain}
\newtheorem{thm}{\protect\theoremname}[section]
  \theoremstyle{definition}
  \newtheorem{definition}[thm]{\protect\definitionname}
  \theoremstyle{plain}
  \newtheorem{lemma}[thm]{\protect\lemmaname}
  \theoremstyle{plain}
  \newtheorem{corollary}[thm]{\protect\corollaryname}
  \theoremstyle{plain}
  \newtheorem{proposition}[thm]{\protect\propositionname}
  \theoremstyle{remark}
  \newtheorem{remark}[thm]{\protect\remarkname}
  \theoremstyle{definition}
  \newtheorem{example}[thm]{Example}

\usepackage{amsfonts}
\usepackage{color}

\usepackage{todonotes}

\newcommand{\R}{ {\mathbb{R}} }
\newcommand{\E}{ {\mathbb{E}} }
\renewcommand{\P}{ {\mathbb{P}} }

\newcommand{\cL}{ {\mathcal{L}} }
\newcommand{\cC}{ {\mathcal{C}}}
\newcommand{\cB}{ {\mathcal{B}}}
\newcommand{\cU}{ {\mathcal{U}}}
\newcommand{\cN}{ {\mathcal{N}}}
\newcommand{\cP}{ {\mathcal{P}}}
\newcommand{\w}{ {\mathbf{w}} }

\theoremstyle{plain}

\renewcommand{\and}{\quad\textrm{ and }\quad}

\renewcommand{\P}{\mathbb{P}}

\newcommand{\cH}{\mathcal{H}}

\newcommand{\ww}{\mathbf{w}}

\newcommand{\product}[2]{\left\langle #1,#2 \right\rangle}

\newcommand{\modif}[1]{\color{black}{#1}\color{black}}

  \providecommand{\corollaryname}{Corollary}
  \providecommand{\definitionname}{Definition}
  \providecommand{\lemmaname}{Lemma}
  \providecommand{\propositionname}{Proposition}
  \providecommand{\remarkname}{Remark}
\providecommand{\theoremname}{Theorem}

\makeatother

\keywords{motion planning, gradient flow, rough paths, regularization by noise}

\subjclass{Primary 34H05; Secondary 60H50,60L20,68T07}

\begin{document}

\title{A gradient flow on control space with rough initial condition}

\author{Paul Gassiat and Florin Suciu}

\address{ Universit\'e Paris-Dauphine, PSL University, UMR 7534, CNRS, CEREMADE, 75016 Paris, France}
\email{gassiat@ceremade.dauphine.fr, florin.suciu@dauphine.psl.eu}

\thanks{PG would like to thank Fr\'ed\'eric Jean for pointers to the motion-planning literature, in particular on the continuation method.}
\thanks{FS's research is supported by the European Union's Horizon 2020 research and innovation programme under the Marie Sk\l{}odowska-Curie grant agreement No.~945332.}

\maketitle

\begin{abstract}
We consider the (sub-Riemannian type) control problem of finding a path going from an initial point $x$ to a target point $y$, by only moving in certain admissible directions. We assume that the corresponding vector fields satisfy the bracket-generating (H\"ormander) condition, so that the classical Chow-Rashevskii theorem guarantees the existence of such a path. One natural way to try to solve this problem is via a gradient flow on control space. However, since the corresponding dynamics may have saddle points, any convergence result must rely on suitable (e.g. random) initialisation. We consider the case when this initialisation is irregular, which is conveniently formulated via Lyons' rough path theory. We show that one advantage of this initialisation is that the saddle points are moved to infinity, while minima remain at a finite distance from the starting point. In the step $2$-nilpotent case, we further manage to prove that the gradient flow converges to a solution, if the initial condition is the path of a Brownian motion (or rougher). The proof is based on combining ideas from Malliavin calculus with \L{}ojasiewicz inequalities. A possible motivation for our study comes from the training of deep Residual Neural Nets, in the regime when the number of trainable parameters per layer is smaller than the dimension of the data vector.
\end{abstract}

\section{Introduction}

\subsection{Description of the problem}

Given smooth vector fields $V_1,\ldots, V_d$ on $\R^n$, an initial point $x \in \R^n$, consider, for $u \in L^2 := L^2([0,1],\R^d)$ the controlled ODE
\begin{equation*}
X_t = x + \int_0^t \sum_{i=1}^d V_i(X_t) u^i_t dt, \quad t \in [0,1].
\end{equation*}
We are then given a target point $y \in \R^n$ and consider the following classical deterministic control problem.
\begin{equation} \label{eq:mp}
 \mbox{ Find } u \in L^2([0,1], \R^d) \mbox{ s.t. } X_1 = y.
\end{equation}
Note that in many situations of interest, the number of vector fields $d$ is smaller than the ambient dimension $n$, so that, at a given time $t$, the system is only allowed to move in a subspace of all possible directions. However, assuming the bracket-generating condition
\begin{equation*}
\forall z \in \R^n, \quad \mathrm{Lie}(V_1,\ldots, V_d)\big|_{z} = \R^n,
\end{equation*}
the well-known Chow-Rashevskii theorem guarantees the existence of a solution to \eqref{eq:mp} (we refer to textbooks on sub-Riemannian geometry e.g.\cite{Mon02,Rif14,ABB19} for a proof of this as well as of the basic facts mentioned in this introduction).
\vspace{2mm}

In fact, the problem \eqref{eq:mp}, known as ''(nonholonomic) motion-planning'' in the control literature, is of considerable importance in applications such as robotics, and numerous procedures have been developed to obtain a solution to it. We refer to \cite[chapter 3]{Jea14} for an overview. Motivated by understanding the gradient descent in deep neural networks (see Subsection \ref{subsec:motiv} below for more details), we consider a non-problem specific gradient flow procedure : let $\mathcal{L}$ be defined by
\[  \cL : u \in L^2([0,1],\R^d) \; \mapsto \;  \modif{\frac 1 2} \left|y - X_1^x(u) \right|^2,\]
so that $\cL \geq 0$ and any zero of $\cL$ is a solution to \eqref{eq:mp}, and given an initial control $u_{init}$, consider the ODE (valued in $L^2([0,1],\R^d)$)
\[ u(0) = u_{init}, \;\; \frac{d}{ds} u(s) = - \nabla_{L^2} \cL(u(s)). \]
We are then interested in the following question : can we find conditions guaranteeing that for $u$ as above,  it holds that
\begin{equation} \label{eq:convgf_intro}
\lim_{s \to \infty} u(s) = u_{\infty} \mbox{ with } \cL(u_{\infty}) = 0,
\end{equation}
\modif{if possible with a quantitative convergence speed ?

A first positive observation is that, under the bracket-generating condition, $\cL$ admits no non-global local minima. Indeed, in that case, the endpoint map $u \in L^2  \mapsto X^x_1(u)$ is open (e.g. \cite[Prop. 1.12]{Rif14}), namely for a given $u$, any $L^2$-neighborhood of $u$ is mapped by $X^1_x$ to a neighborhood of $X_1^x(u)$, and if $X_1^x(u) \neq y$ this neighborhood contains a point which is closer to $y$.}

However, $\cL$ may in general  admit \textbf{critical points which are not minima} : indeed, since the gradient of $\cL$ is seen to be
\begin{equation} \label{eq:nablaL_intro}
 (\nabla_{L^2} \cL)(u) =  (y -X_1^x(u)) \cdot_{\R^n} (\nabla_{L^2} X_1^x)(u) ,
 \end{equation}
this may happen if $u$ is such that the differential of the endpoint map is not surjective. Such controls are well-known to exist (for instance, $u=0$ is always one if $d<n$, since then $Im(dX_1^x)$ is spanned by $V_1(x),\ldots, V_d(x)$), and play an important role in sub-Riemannian geometry (they are typically called \textit{singular controls}).

Another serious problem when  trying to prove convergence is that, since $\cL$ does not contain any cost (or penalization) term, its sub-level sets are not bounded (in fact, it is easy to see that it has zeroes of arbitrarily high norm \footnote{\modif{for instance, if $u$ is any control joining $x$ and $y$, and $v$ is any non-zero control joining $y$ to $y$, then the concatenation of $u$ and $n$ copies of $v$, suitably rescaled to be indexed by $[0,1]$, is a zero of $\cL$ with $L^2$ norm of order $n$.}}), and there is no a priori guarantee that the trajectory will not \textbf{diverge to infinity}.
\vspace{2mm}

In any case, the existence of singular controls means that we cannot hope for convergence results for \emph{all} starting conditions $u_{init}$, but we may still hope for convergence for generic ones, for instance in the sense of a (probability) measure on control space. Indeed, singular controls are rare in this sense, and more precisely, it was shown by Malliavin, as part of his stochastic proof of H\"ormander's hypoellipticity theorem \cite{Mal78}, that
\begin{equation} \label{eq:Mal_intro}
\mbox{ If $u = \dot{B}(\omega)$  where $B$ is a Brownian motion, then, almost surely,  $u$ is non-singular.}
\end{equation}
(Note that $\dot{B}$ is not an $L^2$-function but standard stochastic calculus allows to make sense of all the relevant objects). In our context,  this result means that, if we initialize our gradient descent from such a probabilistic object, we are (almost) never immediately stuck at a saddle-point. Of course, this is much weaker than what we would want, namely that we are also not attracted to such a point (or to infinity) in the long time limit. Nevertheless, the proof of the above statement, and especially of its recent variants in a rough path context clearly highlighting the role of \emph{irregularity} of Brownian motion for the validity of \eqref{eq:Mal_intro}, suggest that it might be possible to use similar ideas to move towards a proof of \eqref{eq:convgf_intro} (with random initial condition). The aim of the present article is to take a first step in this direction. %{\color{red} ***TODO improve wording of final sentences}

\subsection{Main results and methods of proof} \label{subsec:results_intro}
We now describe some of the results that we obtain. We will consider irregular initial conditions $u_{init}$ which are typically not in $L^2$, but it will still make sense to consider the $L^2$ gradient flow (see below for a short discussion on technical details), meaning that the solution will be written at a positive time $s$ as
\[ u(s) = u_{init} + v(s), \]
with $v(s) \in L^2$.
\vspace{2mm}

The first result is of a qualitative nature and shows, in a rather general setting, an advantage of initialising from a rough initial condition.

\begin{thm} \label{thm:qualit_intro} Let $V_1, \ldots, V_d$ be $C^{\infty}_b$ bracket-generating vector fields on $\R^n$. Let $u_{init}=\dot{B}(\omega)$ where $B$ is a Brownian motion. Then, almost surely :

(1) There exists $v \in L^2$ such that $\cL(u_{init} + v) = 0$.

(2)  For any $v$ in $L^2$, $u_{init} + v$ is not a saddle-point, i.e. $\nabla_{L^2} \cL(u_{init} + v) = 0 \Rightarrow \cL(u_{init} + v) = 0$.

(3) If the trajectory $(v(s))_{s \geq 0}$ is bounded in $L^2$, then convergence \eqref{eq:convgf_intro} holds \modif{(and is exponentially fast)}.
\end{thm}

In fact, the results above do not use in a crucial way the Brownian nature of $u_{init}$ and would remain valid for a large class of initialisations (e.g. $\dot{B}^H$ with $B^H$ fractional Brownian motion (fBm) with Hurst index $H \in (\frac{1}{4},1)$). Note that point (1) is similar to the Chow-Rashevskii theorem but not a direct consequence of it, since $u_{init}$ is not in $L^2$. The interpretation of the theorem is as follows: starting our gradient flow from an irregular initial condition means that now all the ''bad'' (saddle) points have been moved to infinity (point (2)), while some zeroes of $\cL$ still remain at a finite distance (point (1)). In particular, this means that, the two obstructions to convergence (saddle points and possible divergence to infinity) now play a similar role in terms of the function $v$, which is made explicit by the convergence criterion (3). 
\vspace{2mm}

Note that the above result, while a clear hint that rough initial conditions may help, does not guarantee convergence as the gradient flow could still diverge to infinity. Our next theorem, and the main result of the paper, shows (almost sure) convergence in a simple (but non-trivial) case.

\begin{thm} \label{thm:conv_intro} Let $V_1, \ldots, V_d$ be $C^{\infty}_b$ bracket-generating vector fields on $\R^n$, with step-$2$ nilpotent Lie algebra, i.e.
\[ \forall i,j,k \in \{1,\ldots,d\}, \;\; \left[ \left[V_i, V_j \right], V_k \right] \equiv 0. \]
Let $u_{init}=\dot{B}(\omega)$ where $B$ is a Brownian motion. Then, almost surely, for any initial and target points $x,y \in \R^n$, convergence \eqref{eq:convgf_intro} holds \modif{(and is exponentially fast)}.
\end{thm}

Note that unlike the previous theorem, this one relies crucially on the precise (ir)regularity of Brownian motion (see \eqref{eq:BMirr_intro} below). In particular, the proof could be adapted to the case of less regular fBm, namely with $H < \frac 1 2$, but not to the case $H > \frac 1 2$. (However, the choice of the Hilbert space in which the gradient is taken is not crucial, as the result also holds if $L^2$ is replaced by the Sobolev space $H^{-\delta}$ for $0\leq \delta < \frac{1}{2}$.) %(TODO also mention elliptic result ?)
\vspace{2mm}

Our continuous-time results can also be seen to imply some asymptotic results on discrete problems. Consider the space  $\cU^L$ of controls who are piecewise-constant controls on each interval $((i-1)/L, iL), i=0,\ldots, L-1$. We write its elements $(u^j_i)_{1\leq j \leq d, 1\leq i \leq L}$ (where $u^j_i$ is the value of the $j$-th coordinate on $((i-1)/L,i/L)$), equipped with the $L^2$-induced norm, namely $\| u \|_{\cU^L}^2 = \frac{1}{L} \|(u^j_i) \|^2_{\ell^2(\R^{dL})}$. We then have the following counterpart of Theorem \ref{thm:conv_intro} above.%  both Theorems above. {\color{red} TODO not clear if we keep this here ?}

\begin{corollary} \label{cor:discrete_intro} Let $V_1, \ldots, V_d$ be $C^{\infty}_b$ bracket-generating vector fields on $\R^n$, with step-$2$ nilpotent Lie algebra. Let $U^{L}$ be a $\cU^L$-valued random variable with coordinates $U^{L,j}_i = \frac{1}{\sqrt{L}} Z^j_i$ where the $Z^j_i$ are independent $\cN(0,1)$.\\
Then, 
%assuming the $V_i$ are smooth and bracket-generating :
%
%(a) it holds that
%\[  \lim_{R \to \infty} \lim_{N \to \infty} \P \left(  \exists \|v\|_{\cU^N} \leq R \mbox{ s.t. } \cL(U^N+v) = 0 \right) = 1, \]
%\[ \forall R > 0  \;\;\lim_{N \to \infty} \P \left(  \exists \|v\|_{\cU^N} \leq R \mbox{ s.t. } \nabla_{\cU^N}\cL(u+v) = 0,  \cL(U^N+v) \neq 0 \right) = 0. \]
%
%(b) If the $V_i$ are step-$2$ nilpotent, 
for the gradient flow on $\cU^L$ defined by
\[ u^L(0)= U^L, \quad \frac{d}{ds} u^L = - \left(\nabla_{\cU^L} \cL\right)(u^L),\]
 it holds that
\[
\lim_{L \to \infty} \P \left(u^L(s) \mbox{ converges as } s \to \infty \mbox{ to } u_{\infty} \in \cU^L \mbox{ with } \cL(u_{\infty}) = 0 \right) = 1.
\]
\end{corollary}

\subsection*{Methods of proof.}

A first important point is that we need to work in an analytic framework which is rich enough to work with irregular trajectories, namely we need to be able to consider equations of the form
\[
X_t = x + \int_0^t \sum_{i=1}^d  V_i(X_r) \left( dw^i_r + dh^i_r \right),
\]
where $w = \int_0^{\cdot} u_{init}$ is our irregular (say, Brownian) initial condition and $h = \int_0^{\cdot} v$ is an arbitrary (varying) element of $\cH=H^1([0,1],\R^d)$. Keeping in mind that in our gradient descent, the value of $h=h(s)$ at $t \in [0,1]$ will typically depend on the whole path $(w_t)_{t \in [0,1]}$, classical It\^o calculus would not be particularly well-suited to the task. We choose instead to work in the framework of Lyons' rough path theory \cite{Lyo98}. Recall that this theory identifies an analytic framework (based on considering iterated integrals of $w$) under which one can make sense of equations driven by an irregular signal $w$ (with a continuous dependence, for the corresponding topology). This achieves a clean separation with the probabilistic aspects, which are then reduced to showing that Brownian motion (or the process of interest) can be placed in this framework. In our context, rough path theory allows us to immediately make sense of all the required objects, and in fact gives us some flexibility w.r.t. both choice of $w$ (which can for instance, be taken as a realisation of many other random processes), and in the choice of the Hilbert space $\cH$ on which we will make the gradient descent (we essentially only need that it satisfies the so-called 'Complementary Young Regularity' assumption w.r.t. regularity of $w$). \modif{
Well-known estimates from rough path theory also allow us to differentiate (in a classical Fr\'echet sense) the It\^o map w.r.t. the translating path $h$ (in the case where $w=0$, this goes back to works of Bismut on ''deterministic Malliavin calculus'' \cite{Bis84}). }
The continuity properties of rough path theory further allow for simple and transparent proofs of convergence results, which are crucial in many places (for instance in proving the Chow-type theorem, or when going from continuous to discrete).
\vspace{2mm}

Let us now comment a bit more precisely on the problem at hand. A first remark is that with $\cL$ described at above, evaluated at $w + h$, it is immediate from \eqref{eq:nablaL_intro} that
\begin{equation} \label{eq:nabla_c_intro}
\left\| \nabla_{\modif{\cH}} \cL(w+h) \right\|^2 \geq c(w+h)^2 \cL(w+h), 
\end{equation}
where
\[
c(w+h)^2 = \inf_{|\xi|_{\R^n} = 1}  \left\| \xi \cdot_{\R^n} \nabla_{\cH} X_1^x(w+h) \right\|_{\cH}^2.
\]
Note that $c^2$ is a classical object in Malliavin calculus, namely the smallest eigenvalue of the so-called Malliavin matrix for the endpoint functional (evaluated at $w+h$). Classical (Malliavin) path-space calculus shows that the gradient on the right-hand side can itself be written in terms of differential equations (driven by $w+h$). We can then follow Malliavin's argument (\modif{which was first adapted to the rough path language by \cite{CF10}, and then in a form closer to the one we use by \cite{HP13,FS13}}). It is an iterative procedure that only requires that $w+h$ satisfies
\begin{equation} \label{eq:tr_intro}
\int_{0}^{\cdot} \sum_{i=1}^d f^i_r  \left( dw^i_r + dh^i_r \right) \;\Rightarrow \; f^1, \ldots, f^d \equiv 0
\end{equation}
(for $f$ in a suitable class of integrands). In our context, this is known to hold if $w+h$ has the so-called \emph{true roughness} property, and since  this property holds (a.s.) for a Brownian path  $w$ and is stable under addition of a sufficiently regular (for instance, bounded variation) path, this leads immediately to the fact that $c(w+h)>0$ (a.s.), and point (2) in Theorem \ref{thm:qualit_intro} above.
\vspace{2mm}

As far as the (a.s.) global convergence result, i.e. Theorem \ref{thm:conv_intro}, we rely on a variant of \L{}ojasiewicz' convergence criterion for gradient flow \cite{Loj63} (see subsection \ref{subsec:Loj}), according to which (keeping in mind \eqref{eq:nabla_c_intro}), it is sufficient to prove a quantitative estimate of the form
\begin{equation} \label{eq:cond_c_intro}
c(w+h)^2 \geq \frac{C(w)}{1 + \|h \|_{\cH}^2}
\end{equation}
for some $C(w) > 0$. We are able to do so in the step-$2$ nilpotent case, using a fine irregularity property of Brownian paths, which is a slight strengthening of the inequality
\begin{equation} \label{eq:BMirr_intro}
\left\| B(\omega) - h \right\|_{L^2} \geq \frac{C(\omega)}{1+ \|h\|_{H^1}}
\end{equation}
for an a.s. positive $C(\omega)$ (see Proposition \ref{prop:BH1L2}). While we are not aware of a previous occurrence of this inequality in the literature, it is very similar to the well-known fact that Brownian motion, as an element in the Besov space $\cB^{1/2}_{2,\infty}$, has its norm a.s. bounded from below by a strictly positive constant (see \cite{Roy93}). This is why our proof crucially relies on specific irregularity of Brownian motion, and would break down for more regular initial conditions.

\vspace{2mm}

We conclude this section by discussing how our work relates with existing literature. In stochastic analysis and Malliavin calculus, the a.s non-degeneracy of the Malliavin matrix (equivalently, non-singularity of white noise \eqref{eq:Mal_intro}) goes back to the original articles \cite{Mal78} (see \cite{Hai11} for a recent exposition), with a proof already using a (stochastic) version of the implication \eqref{eq:tr_intro} above (as already mentioned, the version we use comes from \cite{HP13, FS13}).  In fact, quantitative version of this non-degeneracy (crucial when proving smoothness of densities for hypo-elliptic diffusions) are also known in this context. Typically, these results use as a crucial tool the classical Norris lemma \cite{Nor86}, which was then extended to various contexts, in particular in the recent rough path literature, e.g. \cite{CHLT15}. However, the version we need to prove gradient flow convergence (see \eqref{eq:cond_c_intro}) appears to be stronger than what these results imply, which is why we need to make our (restrictive) nilpotency assumption.

Similar questions have also been considered in the deterministic control literature. An algorithm similar in spirit to our gradient flow is suggested in \cite{DW97} (with no proof for convergence). A slightly more sophisticated ''continuation method'' has been suggested and studied by Sussmann and Chitour \cite{Sus93, CS98, Chi06}. The idea is to find a path $(u(s),0\leq s \leq 1)$ in control space whose endpoints form a prescribed curve, by solving an ODE on control space. They manage to prove well-posedness of this ODE under a so-called ''strong bracket generating'' assumption, \modif{namely
\[
\forall x \in \R^n , \, \forall \theta = (\theta^1, \ldots, \theta^d) \in \R^d,\quad {\rm{span}} \{ V^i(x), [ \theta \cdot V , V^i ](x), i=1,\ldots, d \} = \R^n.
\]
This is a rather strong assumption under which, (as first observed by \cite{Bis84}) all non-null controls are non-singular (so that the functional $\cL$ admits at most one saddle-point). It implies that the Lie algebra of the vector fields is step $2$-generated (but not necessarily nilpotent, unlike what we require in our global convergence result).
} 
Interestingly, the heuristic difficulties for Chitour and Sussman's continuation method (avoiding singular controls and escaping to infinity) are the same as here, and in fact, the crucial step in their proof is also to obtain an inequality of the form \eqref{eq:cond_c_intro} (for $w=0$). This means in particular that our Theorem \ref{thm:conv_intro} has a counterpart in their context (see Remarks \ref{rem:SC1} and \ref{rem:SC2} below for more detailed discussion). More recently, a gradient flow for control problems of the form that we study has been considered by Scagliotti \cite{Sca23}, he considers the case with penalisation which makes the analysis rather different (he obtains convergence to a critical point if the cost is an analytic function).

\subsection{Motivation from machine learning} \label{subsec:motiv}

Recall the setting of supervised learning : we have an unknown function $y : x \in \R^n \mapsto y(x) \in \R^n$ (we assume input and output have the same dimension for simplicity), and we want to find a good approximation in a certain parametrised function class $\left\{ \phi^{\theta}, \;\; \theta \in \Theta\right\}$. This means that we want to find $\theta$ s.t. the (for instance) quadratic loss, for a probability measure $\mu$ on $\R^n$,
\[ \cL(\theta) := \int \left| y(x) - \phi^{\theta}(x)\right|^2 \mu(dx)  \]
is sufficiently small.

We do not assume that $y$ or $\mu$ are known, but only that we have access to a finite sample $(x_i, y_i = y(x_i))_{i=1,\ldots, N}$, so that instead of $\cL$ we minimise the so-called empirical loss
\[ \hat{\cL}(\theta) := \frac{1}{N} \sum_i \left| y_i - \phi^{\theta}(x_i)\right|^2,\]
hoping that it will be close enough to the ''true'' loss (the difference between the two at the chosen $\theta$ is known as the generalisation error).

We are interested in the particular case of deep learning, and more specifically, ResNets, which means that we consider maps of the form $\phi^{\theta}(x) = X_L^{x}(\theta),$ where $L$ represents a fixed number of layers, and $X$ is given by the recursion
\[ X^x_0=x, \;\;\; X^x_{k+1}= X^x_k + \delta_L \sigma( X_k, \theta_k). \]
Above, $\delta_L$ is a fixed constant (although it could also be considered as an additional parameter), and $\sigma:\R^n \times \R^d \to \R^n$ is a fixed nonlinear function. The parameter set is therefore $\Theta = \{ (\theta_k)_{k=0,\ldots, L-1}\vert\,\theta_k\in\R^d,\forall k=0,\dots,L-1\} = (\R^d)^L$.

In the regime where $L \to \infty$ and $\delta_L \sim \frac{1}{L}$, the above recursion can be interpreted as the Euler scheme for the ODE $\dot{X}_t = \sigma(X_t, \theta_t)$, which can then be analysed via dynamical systems / continuous analysis methods. This point of view (sometimes referred to as 'neural ODE') has been proposed by several authors \cite{E17, HR17, CRBD18}. 

However, it has also been argued that ResNets should be understood not as ODEs but as \emph{Stochastic} differential equations (SDE), by taking different values for the scaling parameter $\delta_L$ \cite{PF20,CCRX21}. In particular, if the parameters act linearly layer-wise, i.e., the recursion is written $X_{k+1} = X_k + \delta_L \sigma(X_k) \theta_k$, then it has been shown \cite{Hay22,MFBV25} that, if the $\theta_k$ are sampled from i.i.d. Gaussians, the scaling $\delta_L= \frac{1}{\sqrt{L}}$ is the only one to lead to a non-trivial scaling limit, namely the It\^o SDE $dX_t = \sigma(X_t) dB_t$. Note, however, that these works only concern the initialisation (or the already trained network) and do not investigate if the scaling limit is preserved through training. {Nonetheless, in a recent line of works \cite{BNLHP24,YYZH23,CN24}, it has been shown that ResNets under $\frac{1}{\sqrt{L}}$ scaling exhibit hyperparameter transfer and feature learning phenomena during training in the large-depth regime, underlying the practical significance of this regime.} \modif{It has also been shown that rough path bounds (conveniently adapted to the setting of discrete recursions) could be used to measure stability of ResNets in this type of scaling regimes \cite{BFT22}}.% \modif{\cite{BFT22} show that bounds from discrete rough-path theory can be used to obtain stability estimates for ResNets under scaling regimes where weights behave like Brownian-type paths.}

The limiting deterministic ODEs have also been investigated from the control perspective in several works. The case of affine systems $dX_t = \sigma(X_t) \theta_t dt$ (which are the relevant ones for SDE initialisation), 
 has been treated by Agrachev and Sarychev \cite{AS22}, and Cuchiero et al. \cite{CLT20}. Note that the problem is then very similar to the one studied in this paper. In particular, finding a zero of the empirical loss is the $N$-point generalisation of \eqref{eq:mp} : given $(x_i,y_i)_{i=1,\ldots, N}$, 
\begin{equation*}  
 \mbox{ Find } u \in L^2([0,1], \R^d) \mbox{ s.t. } X_1^{x_i}(\theta) = y_i, \;\;\; i=1, \ldots, N.
\end{equation*}
(Note that the important point is that the control $\theta$ does not depend on $i$). \cite{CLT20} show in particular that it is possible to find $5$ vector fields $\sigma_1,\dots,\sigma_5$ s.t. the above has a solution for \emph{arbitrarily high} $N$ (as long as the state dimension $n$ is greater than $2$ and the $x_i, y_i$ are all distinct).

Given that the above results imply that, at least in the large depth limit, there exist parameters for the ResNets which achieve zero (or very small) empirical loss, an important and natural question is then whether these can be found by the usual gradient-descent type training procedures, and whether the choice of initialisation (randomness, scaling) matter in this phase.

Note that while there are already some interesting results in the literature on the convergence of gradient descents for ResNets (see e.g. \cite{CRX22,BPV22,MWSB24} for recent contributions), they typically assume some relation between number of parameters $d$ per layer (e.g. via layer width) and number of data points $N$, which in particular means ''instantaneous'' controllability, so that the depth of the network does not seem to be important for the convergence (in contrast to the above mentioned controllability results, where continuous depth is crucial).

While most of our results mentioned in section \ref{subsec:results_intro} above do not apply directly to ResNets, we believe the ideas in this paper may be relevant in this context, and we intend to explore the connection further in future research. More precisely, our contributions in this setting are the following:
\begin{itemize}
\item propose rough path theory as a convenient and flexible technical framework to describe scaling limits of ResNets, not only at initialisation but also through training,
\item identify qualitative advantages of the ''stochastic'' ($\frac{1}{\sqrt{L}}$ and i.i.d. weights) scaling, via the non-degeneracy result Theorem \ref{thm:qualit_intro}, which actually has a counterpart for the empirical or true loss, see Proposition \ref{prop:ncN}), according to which saddle-points are asymptotically infinitely far from the initial condition.
 \item further provide a complete proof of convergence in a simple (but non-trivial) case with this scaling (Theorem \ref{thm:conv_intro}), using the actual irregularity of weights combined with a \L{}ojasiewicz criterion which degenerates at infinity.
\end{itemize}
Note that the last two items crucially rely on depth of the network, since they apply to genuinely sub-Riemannian situations.

\subsection{Organisation of the paper}

The rest of this article is organised as follows. In section \ref{sec:prelim}, we discuss preliminary requisites, including rough path theory and \L{}ojasiewicz-type criteria for convergence of gradient flows. In section \ref{sec:defGF}, we record the setting of gradient flow on control space setting that we consider in the rest of the paper. In section \ref{sec:no-saddles}, we obtain non-existence of saddle-points under the so-called true roughness condition (this implies Theorem \ref{thm:qualit_intro} (2) above). In section \ref{sec:CR-rough}, we obtain a general Chow-Rashevskii result with rough drift (which in particular implies Theorem \ref{thm:qualit_intro} (1)). In section \ref{sec:conv}, we prove convergence of gradient flow in both the elliptic case (for general initial condition) and the step-$2$ nilpotent case, started from Brownian motion (Theorem \ref{thm:conv_intro}). In section \ref{sec:cont}, we discuss continuity with respect to initial condition and underlying Hilbert space, and we obtain in particular discretised versions of our convergence theorem (Corollary \ref{cor:discrete_intro}). Finally, in section \ref{sec:num}, we provide some numerical experiments where we vary the regularity of the initialisation, and observe faster convergence for the more irregular paths. 

\section{Preliminaries} \label{sec:prelim}

\subsection{Notations}
Throughout this paper, we fix the \textit{control} dimension $d \geq 1$. Unless otherwise specified, all the function spaces that we consider will be from $[0,1]$ to $\R^d$.

Given a path $f : [0,1] \to \R^d$, we denote the increments $f_{s,t} := f_t-f_s$ for $s, t \in [0,1]$.

For $q \geq 1$, $C^{q-var}$ is the set of continuous paths $f$ (from $[0,1]$ to $\R^d)$ such that 
\[
\| f \|_{q-var}^q := \sup_{0 = t_0 \leq\ldots \leq t_m=1} \sum_{i=0}^{m-1} \left| f_{t_i,t_{i+1}}  \right|^q
\]
is finite.

We will write $H^0=L^2$. For $\delta \in (0,1]$, we will also consider the Sobolev-Slobodeckij seminorms
\[
\left\| f \right\|_{H^\delta}^2 := \int_{0 \leq s \leq t \leq 1} \frac{|f_{s,t}|}{|t-s|^{1+2\delta}} dt dt, \quad\mbox{for } 0< \delta < 1,
\]
and 
\[
\left\| f \right\|_{H^1}^2 = \int_0^1 |\dot{f}_t|^2 dt.%, \quad \mbox{if } f = \int_0^{\cdot} \dot{f}.
\]
For $\delta \in (1/2,1]$ any $f$ with finite $H^\delta$-norm is continuous (up to modification), so that pointwise evaluation makes sense. We then let 
\[
H^\delta_0 = \left\{ f : f_0 = 0 \mbox{ and } \| f \|_{H^{\delta} } < \infty\right\}. 
\]
These spaces are Hilbert spaces when equipped with the $\left\| \cdot \right\|_{H^\delta}$ norms.

\subsection{Rough path theory} \label{subsec:rp}

We recall some standard facts and notations from rough path theory. (Details can be found in \cite{FV10}).

For $N \geq 1$, let $G^N(\R^d) \subset \oplus_{k=0}^N (\R^d)^{\otimes k}$ be the step-$N$ free nilpotent Lie group over $\R^d$. Given a $G^N(\R^d)$-valued path $\ww$, we define its increments $\ww_{s,t}:= \ww_t \otimes \ww_{s}^{-1}$ for $s, t \in [0,1]$. For $p > 1$, we let $\cC_g^{p-var} =\cC_g^{p-var}([0,1],\R^d)=  C^{p-var} \left([0,1], G^{\lfloor p \rfloor}(\R^d) \right)$ be the set of (weakly) geometric $p$-variation rough paths. The rough path "norm" is defined by
\[
\| \ww\|^p_{p-var} =  \sup_{0 = t_0 \leq\ldots \leq t_m=1} \sum_{i=0}^{m-1} \left| \ww_{t_i,t_{i+1}}  \right|^p
\]
where $|\cdot|$ is the Carnot-Carath\'eodory metric on $G^N(\R^d)$.

Given $\alpha \in (0,1)$, we let $\cC_g^{\alpha} \subset \cC_g^{(1/\alpha)-var}$ be the set of $\alpha$-H\"older rough paths, i.e., the set of $\ww$'s s.t. $\sup_{s \neq t} \frac{|\ww_{s,t}|}{|t-s|^{\alpha}} < \infty$. 

We also let $\rho_{p-var}$ be the inhomogeneous rough path distance 
\[ \rho_{p-var}(\ww, \ww') =  \sum_{k=1}^{\lfloor p \rfloor} \sup_{0 = t_0 \leq \ldots\leq t_m=1} \left( \sum_{i=1}^m \left| \pi_k(\ww_{t_i,t_{i+1}} - \ww'_{t_i,t_{i+1}}) \right|^{\frac{p}{k}} \right)^{\frac{k}{p}}
\]
where $\pi_k$ is the projection $G^N(\R^d) \to (\R^d)^{\otimes k}$.

For any path $x$ of finite $1$-variation, the canonical rough path lift is denoted by $S(x)$, and defined by
\[
S(x)_t = \left( x_{0,t}, \int_0^t x_s \otimes dx_s, \ldots, \int_{0 \leq s_1 \leq \ldots \leq  s_{\lfloor p \rfloor} \leq t} d x_{s_{\lfloor p \rfloor}} \otimes \cdots \otimes dx_{s_1} \right)
\]
 Then any element $\ww$ of $\cC_g^{p-var}$ can be obtained as a limit of such lifts \modif{(see \cite[Prop. 8.12]{FV10})}, namely, there exists a sequence $(w_n)$ in $C^{1-var}$ s.t.
\begin{equation} \label{eq:limWn}
\forall t \in [0,1], \;\; \lim_{n} S(w^n)_t =\ww_t, \mbox{   and }  \sup_n \| S(w^n) \|_{p-var} < + \infty.   
\end{equation}

For $1 \leq q < 2$ ($q\leq p$) s.t. $\frac{1}{q} + \frac{1}{p} > 1$, for $\w \in \cC_g^{p-var}([0,1],\R^d)$ and $h \in C^{q-var}([0,1],\R^{d'})$, the \emph{Young pairing} \modif{(\cite[Sec. 9.4]{FV10})} of $\w$ and $h$ is denoted by $(\w,h)$. It is the element of $ \cC_g^{p-var}([0,1],\R^{d+d'})$ such that, if $(w_n)$ and $(h_n)$ are $1$-variation paths which converge respectively to $\w$ and $h$ in  the sense of \eqref{eq:limWn}, then $(\w,h)$ is the pointwise limit of $S((w_n,h_n))$. 
When $d'=d$, the \emph{Young translation} of the $p$-rough path $\ww$ by  $h \in C^{q-var} = C^{q-var}([0,1],\R^d)$ is denoted $\ww + h$. It can again be defined by approximations, namely, if  $w^n$ are $C^{1-var}$ elements satisfying \eqref{eq:limWn}, then $\ww + h = \lim_n S(w^n + h)$. 
 
Young translation is locally Lipschitz continuous in the sense that, for any $\ww$, $\ww'$ in $\cC^{p-var}_g$ and $h, h'$ in $C^{q-var}$, for any $M>0$ there exists a constant $C(M)>0$ s.t.
 \begin{align*} \label{eq:ContYoungTransl}
 \| \ww \|_{p-var} , \| \ww' \|_{p-var}, &\| h \|_{q-var}, \| h' \|_{q-var} \leq M \\
 &\Rightarrow  \;\; \rho_{p-var} \left( \ww + h, \ww' + h' \right) \leq C(M) \left( \rho_{p-var} (\ww, \ww') + \|h - h'\|_{q-var} \right),
 \end{align*}
 \modif{see \cite[Theorem 9.33]{FV10}}.
 
 Let $V_1, \ldots, V_d$ be $C^{\infty}_b$ vector fields on $\R^n$, and $x_0 \in \R^n$. For any rough path $\ww$, we consider the rough differential equation (RDE)
 \begin{equation}
 X_t = x_0 + \int_0^t \sum_i V_i(X_s) d \ww^i_s.
 \end{equation}
 $X=X(\ww)$ can again be defined by approximations, namely, by $X = \lim X^n$ where each $X^n$ is the solution along a $C^{1-var}$ path $w^n$, where $w^n$ satisfies \eqref{eq:limWn} \modif{(see \cite[Sec. 10.3]{FV10})}.
 
 We then consider the map 
 \[(\ww, h) \in \cC^{p-var}_g \times C^{q-var} \mapsto X_1 ( \ww + h),\]
 where again $\frac{1}{q} + \frac{1}{p} > 1$. This map is Fr\'echet differentiable w.r.t. its second coordinate \modif{(see \cite[Theorem 11.6]{FV10})}, and the derivative can be written as
 \[ d_{h}X_{1}(\ww) [k] := \lim_{\varepsilon \to 0} \frac{X_{1}(\ww +(\varepsilon k))}{\varepsilon} = \int_0^1 \sum_i DV_i(X_s) J_{1 \leftarrow s}{(\ww)} dk^i_s, \]
 where the last integral is a Young integral and $ J_{t \leftarrow s}$ is the Jacobian of the RDE flow between $s$ and $t$, namely the ($\R^{n \times n}$-valued) solution to the linear RDE
\begin{equation*}\label{eq:rde-jacobian} 
J_{t\leftarrow s} = Id_{\R^n} + \int_s^t  \sum_{i=1}^d DV_i(X_u) J_{u\leftarrow s} d\ww_u^i. 
\end{equation*}

Note that the pair $(X,J)$ is solution to a RDE driven by $\ww$, and in particular is a continuous function of $\ww$ in rough path metric. This implies, in particular, that the derivative above is locally Lipschitz in the sense that for any $M>0$, there exists $C>0$ such that, for $\ww, \ww' \in \cC^{p-var}_g$ with $\| \ww \|_{p-var} , \| \ww' \|_{p-var} \leq M$, it holds that
 \begin{equation*} \label{eq:DXLip}
  \left| d_{h}X_{1}(\ww) [k] - d_{h}X_{1}(\ww') [k] \right| \leq C \rho_{p-var}(\ww, \ww') \|k\|_{q-var}.
 \end{equation*}
 
\subsection{Gradient flow bounds} \label{subsec:Loj}

In this section, we fix a Hilbert space $H$ with norm $\left|\cdot\right|$, and we consider a map $L : H\to \R_+$ which is $C^{1,1}_{loc}$, in the sense that it is Fr\'echet-differentiable with gradient $\nabla L$, and
\begin{equation} \label{eq:LC11}
 \forall R >0, \exists C_R>0, \forall x,y \in H \mbox{ with } |x|,|y| \leq R, \quad \left| (\nabla L)(x)- (\nabla L)(y) \right| \leq C_R |x-y|.
\end{equation}
We consider the associated gradient flow
\begin{equation} \label{eq:gfH}
\dot{x}(t) = -\nabla L(x(t)), \;\;\; x(0) = x_0 \in H.
\end{equation}
\begin{proposition} \label{prop:gfdef}
Under the above assumptions, for any initial condition $x_0$, the gradient flow \eqref{eq:gfH} admits a unique global solution $x=(x(t))_{t\geq0}$.
\end{proposition}

\begin{proof}
Local existence and uniqueness follow from standard Cauchy-Lipschitz theory, due to the regularity assumption on $L$. To prove that the solution is global, we use Cauchy-Schwarz inequality to obtain the a priori bound. For any $t\geq 0$ (at which the solution is defined),
\[  \left| x(t) - x_0 \right| \leq \int_0^t |\dot{x}(s)| ds \leq \left( \int_0^t  \left| \nabla L(x(s)) \right|^2 ds \right)^{1/2} t^{1/2} = \left( L(x_0) - L(x(t)) \right) t^{1/2}, \]
and since $L \geq 0$ it follows that the solution cannot blow up in finite time.
\end{proof}

Recall that a nonnegative function $L$ satisfies the classical Polyak-\L{}ojasiewicz criterion if $| \nabla L | \geq c \sqrt{L}$ for some $c>0$, in which case gradient flow trajectories converge to minima. We will use the following extension of this fact, where the constant is allowed to degenerate at infinity. The argument is a simple variant of \L{}ojasiewicz's argument \cite{Loj63,Loj82} (see also \cite{KMP00} for a presentation).%, where the constant $c$ in the inequality is allowed to decrease as $|x|$ increases.

\begin{proposition} \label{prop:Loj}
Let $L : H \to \R_+$ be $C^2$, and let $c:\R_+\to  \R_+$ be nonincreasing and s. t.
\begin{equation} \label{eq:crLoj}
 \forall r \geq 0, \quad c(r) \leq  \inf \left\{  \frac{\left| \nabla L(x) \right| }{\sqrt{L(x)}}, \;\; \; |x| \leq r\right\} 
\end{equation}
(with convention $\frac{0}{0}=\infty$), and denote
\[ C(z) = \int_0^z {c}(v) dv. \]
Then, if $(x_t)_{t\geq 0}$ is a solution to the gradient flow \eqref{eq:gfH}, it holds that
\begin{equation} \label{eq:LojBound}
 \forall t \geq 0, \quad C \left( |x_0| +  \int_0^t  |\dot{x}_s| ds \right)   + 2 \sqrt{L}(x_t) \leq C(|x_0|)+  2 \sqrt{L}(x_0).
  \end{equation}
  In particular, if it holds that
  \begin{equation} \label{eq:CR}
   \mbox{ For some } R>0, \quad C(R) > C(|x_0|)+  2 \sqrt{L}(x_0),
  \end{equation}
   then
   \[  
   \exists x_{\infty} \in H \mbox{ with } |x_{\infty}| < R, \; L(x_{\infty}) =0, \mbox{ and }\lim_{t \to \infty} x(t) = x_{\infty}.
  \]
  More precisely, we then have
\begin{equation} \label{eq:decayL}
\forall t \geq 0, \quad L(x(t)) \leq L(x_0) e^{-c(R)^2 t}, 
\end{equation}  
\begin{equation} \label{eq:gfconvRate}
\forall t \geq 0, \quad \left| x(t) - x_{\infty} \right|  \leq \frac{2 \sqrt{L}(x_0)}{c(R)} e^{-c(R)^2 t/2}. 
\end{equation}
\end{proposition}

\begin{proof}
To prove \eqref{eq:LojBound}, we differentiate the quantity appearing on the l.h.s. and obtain
\begin{align*}
 |\dot{x}_t | c \left(  |x_0| + \int_0^t  |\dot{x}_s| ds \right) -  \frac{\left| \nabla L(x_t) \right|^2}{\sqrt{L}(x_t)}  
\end{align*}
which is nonpositive since the second term is 
\[ | \dot{x}_t| \frac{\left| \nabla L(x_t) \right|}{\sqrt{L}(x_t)} \geq | \dot{x}_t|  c(|x_t|) \geq |\dot{x}_t | c \left(  |x_0| + \int_0^t  |\dot{x}_s| ds \right). \]

If \eqref{eq:CR} holds, then the trajectory $(x_t)_{t \geq 0}$ has finite length and is bounded by $R$, and, in particular, it converges to a limit $x_\infty$ with $|x_{\infty}| \leq R$. 

The inequality \eqref{eq:decayL} is an immediate consequence of Gronwall's Lemma, and it then further implies $L(x_\infty)=0$. Finally, \eqref{eq:gfconvRate} follows from
\begin{align*}
c(R) \left| x_t - x_{\infty} \right| \leq C\left( |x_t| + \int_t^\infty |\dot{x}_s| \right) - C\left( |x_t| \right) \leq 2 \sqrt{L(x_t)},
\end{align*}
where the first inequality follows from \eqref{eq:LojBound} (starting at time $t$ instead of $0$).
\end{proof}

We now state some simple extensions or special cases of the above. We first note that if the \L{}ojasiewicz inequality is satisfied on bounded sets, then convergence of the gradient flow is equivalent to its boundedness.

\begin{proposition} \label{cor:GFbdd}
In the setting of the previous theorem, assume that $c(r) > 0$ for all $r>0$. Then, if the gradient flow trajectory $(x_t)_{t\geq 0}$ is bounded in $H$, it converges to a minimum.
\end{proposition}

\begin{proof}
This follows from noticing that, in the proof of Proposition \ref{prop:Loj}, the inequality \eqref{eq:crLoj} only needs to hold on the gradient trajectory $(x(t))_{t \geq 0}$. If the latter is bounded by $\rho$, we can simply take $c \equiv c(\rho) > 0$, for which $C(z) = c z$ clearly satisfies \eqref{eq:CR}.
\end{proof}

We also have as a special case the following local convergence criterion.

\begin{corollary} \label{cor:LocLoj}
Let $L:H\to \R_+$ be $C^2$, and $x_0 \in H$ be s.t. for some $R,c>0$,
\[ \forall x \in H, \;\; |x - x_0 | \leq R\;\; \Rightarrow \;\; |\nabla L(x)| \geq c \sqrt{L(x)}, \]
where in addition it holds that
\begin{equation*} \label{eq:LocCrit}
2 \sqrt{L}(x_0) \leq c R,
\end{equation*}
then the solution of \eqref{eq:gfH} satisfies $\lim_{t \to \infty} x_t = x_{\infty}$ in $H$ with $|x_\infty| < R$ and $L(x_{\infty}) = 0$.

%In particular, if $\hat{x}_\infty$ is such that $L(\hat{x}_\infty) = 0$ and, for some $r,c>0$,
%\[ \forall x \in H, \;\; |x - \hat{x}_\infty | \leq r \;\; \Rightarrow \;\; |\nabla L(x)| \geq c \sqrt{L(x)}. \]
%Then, there exists $r'>0$ such that, for any $x_0 \in H$ with $|x_0 - \hat{x}_{\infty}| \leq r'$, there exists $x_{\infty}$ (depending on $x_0$) s.t. $L(x_{\infty}) = 0$, and
\end{corollary}

\begin{proof}
By translating we may assume $x_0=0$, and we then apply Proposition \ref{prop:Loj} with
\[ c(r) = \begin{cases} c, & r \leq R \\0, & r > R \end{cases}.\]\end{proof}
For future reference, we record the special case where $c(r)$ decays as $\frac{1}{1+r}$ at infinity, which is the one that we will use in our global  convergence results below.
\begin{corollary} \label{cor:c1r}
Let $L : H \to \R_+$ be $C^2$, and assume that, for some $c > 0$, it holds that
\[
\forall x \in H, \;\; \left| \nabla L (x) \right| \geq c \frac{ \sqrt{L(x)}}{1 + |x|}.
\]
Then, if $(x_t)_{t\geq 0}$ is a solution to the gradient flow \eqref{eq:gfH}, it holds that   
\[  
   \exists x_{\infty} \in H \mbox{ with } |x_{\infty}|  \leq  e^{\frac{2}{c} \sqrt{L}(0)} -1 , \; L(x_{\infty}) =0, \mbox{ and }\lim_{t \to \infty} x(t) = x_{\infty},
  \]
  \end{corollary}

\begin{proof}
This follows from taking $c(r) = c/(1+r)$ in Proposition \ref{prop:Loj}.
\end{proof}

\begin{remark} From the estimates in Proposition \ref{prop:Loj}, we see that in the setting of Corollary \ref{cor:c1r}, it holds that 
\begin{equation*} 
\forall t \geq 0, \quad L(x(t)) \leq L(0) e^{- c e^{ - \frac{2}{c} \sqrt{L}(0)} t} , 
\end{equation*}  
\begin{equation*} 
\forall t \geq 0, \quad  \left| x(t) - x_{\infty} \right|  \leq \frac{2}{c}  \sqrt{L}( 0 ) e^{\frac{2}{c} \sqrt{L}(0)} e^{-c e^{ - \frac{4}{c} \sqrt{L}(0)}  t/2}. 
\end{equation*}
\end{remark}

\begin{remark}
The Polyak-\L{}ojasiewicz inequality is the simplest case of so-called Kurdyka-\L{}ojasiewicz inequalities \cite{Kur98}. These are written as
\[ \left| \nabla (K \circ L ) \right| \geq c > 0 \]
where $K : [0, \infty) \to \R_+$ is  $C^1$ on $(0,+\infty)$. (The special case $K = \vert\cdot\vert^{1-\theta}$ with $\theta \in (0,1)$ corresponds to the \L{}ojasiewicz inequalities $|\nabla L | \geq c L^{\theta}$.) 

It is easy to see that the convergence criterion above could be adapted to this more general form (simply by replacing $\sqrt{\cdot}$ by $K$ in the proof), but for simplicity we stick with this simplest case in the rest of the paper. 
\end{remark}

\section{Formulation of the problem} \label{sec:defGF}

To formulate the problem, we fix:
\begin{itemize}
\item a geometric rough path $\w$ $\in$ $C_g^{p-var}$  (the \emph{initial condition}).
\item a Hilbert space $\cH$, which, throughout, we assume to satisfy Complementary Young Regularity, namely $\cH \subset C^{q-var}$, where $\frac{1}{p} + \frac 1 q > 1$.
\end{itemize}
As discussed in Section \ref{subsec:rp}, under this assumption, the Young translation $\w + h$ is well-defined for each $h \in \cH$.

For fixed $x \in \R^n$ and $C^{\infty}_b$ $\R^n$-valued vector fields $V^i$, with $i=1,\ldots, d$, for any $h \in \cH$, we consider $X:= X(\w + h) $ the solution to the RDE
\begin{equation}\label{eq:rde-h}
dX_t = \sum_i V_i (X_t) d (\w + h)^i_t, \;\; X_0=x.
\end{equation}
We fix a \emph{cost} function $g \in C^2(\R^n)$, and consider the \emph{loss} functional
\begin{equation}\label{eq:def-loss}
\mathcal{L}_{\w}(h) = g\left( X_1(\w + h)\right). 
\end{equation}

Recall that, by rough path theory, the map $h \mapsto X_1(\w + h)$ is $C^1$, so that by the chain rule $\mathcal{L}_{\w}:\cH \to \R$ is $C^1$, and its gradient satisfies
\begin{equation} \label{eq:nablaL}
\nabla_{\cH} \cL_{\ww} = (\nabla g)(X_1) \cdot_{\R^n} \nabla_{\cH} X_1.
\end{equation}

We will then study the gradient flow $(h(s))_{s \geq 0}$ defined by
\begin{equation} \label{eq:roughgf}
h(0)=0,  \quad \forall s\geq 0,\; \frac{d}{ds} h(s)= - \nabla_{\cH} \cL_{\w}(h(s)) 
\end{equation}
\begin{proposition}
Under the above assumptions, \eqref{eq:roughgf} admits a unique solution, which is global in (forward) time.
\end{proposition}

\begin{proof}
We check \eqref{eq:LC11}. It follows from standard rough path results (see subsection \ref{subsec:rp}) and \eqref{eq:nablaL}  that
\begin{align*}
\| \nabla_{\cH} \cL_{\w}(h) & - \nabla_{\cH} \cL_{\w}(h') \|_{\cH} \\
\begin{split}
	&\leq C\left( | X_1(\w + h) |, | X_1 (\w +h') | \right) \big( | X_1(\w + h) - X_1 (\w +h') | \\
	&\qquad\qquad\qquad\qquad\qquad\qquad\qquad\qquad\qquad+ \|\nabla_{\cH} X_1(\w+h) -  \nabla_{\cH} X_1(\w+h') \|_{\cH}\big) 
\end{split}
\\
&\leq C\left( \| \w \|_{p-var}, \| h \|_{q-var}, \| h' \|_{q-var} \right) \rho_{p-var} (\w + h, \w + h') \\
&\leq C\left( \| \w \|_{p-var}, \| h \|_{q-var}, \| h' \|_{q-var} \right) \| h - h'\|_{q-var} \\
&\leq C\left( \| \w \|_{p-var}, \| h \|_{\cH}, \| h' \|_{\cH} \right) \| h - h'\|_{\cH},
\end{align*}
where the functions $C(\cdot)$ are finite but may vary from line to line, and where we have used complementary Young regularity in the final inequality. The result then follows from Proposition \ref{prop:gfdef}.
\end{proof}

Recall the formula for the gradient of $X_1$ (cf subsection \ref{subsec:rp}) : for any $k \in \cH \subset C^{q-var}([0,1],\R^d)$ 
\begin{equation}\label{eq:formula-eq-gradient}
\left\langle \nabla_{\cH} X_1, k \right\rangle_{\mathcal{H}} = \sum_{i=1}^d \int_0^1  J_{1 \leftarrow t} V_i(X_t) dk^i_t,
\end{equation}
where $ J_{1 \leftarrow t}$ is the Jacobian of the RDE flow between $t$ and $1$, namely the ($\R^{n \times n}$-valued) solution to
\begin{equation}\label{eq:formula-eq-jacobian}
J_{t\leftarrow t} = Id_{\R^n}, \;\; d_u J_{u\leftarrow t} = \sum_{i=1}^d DV_i(X_u) J_{u\leftarrow t} d(\w + h)_u^i. 
\end{equation}

Recall that, writing $X_1=(X^k_1)_{k=1,\ldots, n}$, the Malliavin matrix (for the functional $X_1$) is given by
\begin{equation*}
\mathcal{M}_{\w}(h) = \left( \left\langle \nabla_{\cH} X^i,  \nabla_{\cH} X^j  \right\rangle_{\cH}\right),_{1\leq i, j \leq n}.
\end{equation*}
This is a symmetric and nonnegative matrix, and consider its smallest eigenvalue $c_{\w}(h)$, namely
\begin{equation} \label{eq:defc}
c_{\w}(h)^2 := \inf_{|\xi|_{\R^n}=1} \mathcal{M} \xi \cdot_{\R^n} \xi = \inf_{|\xi|_{\R^n}=1}   \left\| \xi \cdot_{\R^n} \nabla_{\cH} X_1 \right\|_{\cH}^2.
\end{equation}

The main tool in our analysis will be the observation below, that if $g$ satisfies a \L{}ojasiewicz inequality, so does $\cL_{\ww}$ locally, with  constant proportional to $c_{\w}(h)$. 

\begin{proposition} \label{prop:LojL}
Let $\w, \cH$ as above. Then it holds that
\[ \forall h \in \cH, \quad \left\| \nabla \cL_{\ww}(h) \right\|_{\cH} \geq c_{\w}(h)   | (\nabla g)(X_1( \w+h)) |_{\modif{\R^n}}  .\]
 In particular, if for some $c_g > 0$, it holds that $| \nabla g | \geq c_g \sqrt{g}$ on $\R^n$, then
\[ \forall h \in \cH, \quad \left\| \nabla \cL_{\ww}(h) \right\|_{\cH} \geq  c_g c_{\w}(h) \sqrt{ \cL_{\ww}(h)}.\]
\end{proposition}

\begin{proof}
This is immediate from the definition of $c_{\w}$ and the expression \eqref{eq:nablaL} for $\nabla_{\cH} \cL_{\ww}$. 
\end{proof}

For $f \in C^{p-var}$, let 
\begin{equation} \label{eq:defHvee}
\left\| f \right\|_{\cH^{\vee}} = \sup_{ \| h \|_{\mathcal{H}} \leq 1} \int_0^1 f(t) dh(t).
\end{equation}
This is a semi-norm on $C^{p-var}$, and we will frequently assume that it is actually a norm (this is a non-degeneracy assumption on $\cH$, which holds if for instance $C^\infty \subset \cH$ ).

(Although we will not use this, it may be useful to the stochastic analysis reader to notice that, if $\cH$ is the Cameron-Martin space associated to a Gaussian process $X$, then $\cH^{\vee}$ is the norm associated to the corresponding Paley-Wiener integral, namely $\|f\|_{\cH^\vee} = \| \int_0^1 f_t dX_t \|_{L^2(\P)}$).

We now record a useful expression for the quantities $\left\| \xi \cdot \nabla_{H^1} X_1 \right\|_{\mathcal{H}}$. 
\begin{lemma}
In the above setting, for any $\xi \in \R^n$, it holds that
\begin{equation} \label{eq:FormulaGradient}\left\| \xi \cdot \nabla_{\cH} X_1 \right\|_{\mathcal{H}} =  \left\|   t \mapsto \left( \langle J_{1 \leftarrow t} V_i(X_t)  \cdot \xi \right)_{i=1,\ldots,d} \right\|_{\cH^{\vee}}. 
\end{equation}
\end{lemma}

\begin{proof}
This follows from \eqref{eq:formula-eq-gradient} and the fact that  $\left\|  \cdot \right\|_{\mathcal{H}} = \sup_{\| k \|_{\cH} \leq 1} \left\langle \cdot, k\right\rangle_{\cH}$.	
\end{proof}

%
%
%Note that, in the case where $\cH=H^1$, it follows that $\cH^{\vee} =L^2([0,1],\R^d)$ which yields
%\[
%\left\| \xi \cdot \nabla_{} X_1 \right\|_{H^1}^2 = \sum_{i=1}^d \int_0^1 \left \langle J_{1 \leftarrow t} V_i(X_t), \xi\right \rangle^2 dt.
%\]
%
%We will use that, for any vector field $W$,
%\begin{equation}
%\left \langle  J_{1 \leftarrow t} W (X_t), \xi \right\rangle =  \left \langle W(X_1),  p \right\rangle  - \sum_j \int_t^1 \left \langle J_{1 \leftarrow u}   [W, V_j](X_u),  \xi \right\rangle  d (w+ h)^j_u.
%\end{equation}

\section{No saddle-points under true roughness} \label{sec:no-saddles}

 \modif{
In this section, we let $ \ww \in C_g^{p-var}$. Our result will be valid under some roughness conditions on the underlying path, which we now introduce. 

\begin{definition} \label{def:TR}
 Given a control function\footnote{\modif{A function $\omega: \{(t,s) \in [0,1]^2; t\leq s \} \to \R_+$ is a control function if $\omega(t,s) \geq \omega(t,u)+\omega(u,s)$ for all $t\leq u \leq s$.}} $\omega$ s.t. $\| \ww \|_{p-var;[t,s]}^p \leq \omega(t,s)$ for all $t \leq s$ in $[0,1]$, and $\beta \in (\frac{1}{p}, \frac{2}{p}]$, we say that $\ww$ is $(\omega,\beta)$-rough at $t$ if
\begin{equation} \label{eq:TRpvar}
 \forall  0\neq \zeta \in \R^d , \quad \limsup_{s \downarrow t} \frac{  \left|w_{t,s} \cdot \zeta \right|}{\omega(t,s)^{\beta}} = +\infty,
\end{equation}
where $w$ is the level-$1$ component of $\ww$.

In the case where the above holds with $\omega(t,s)=|s-t|$, we say that $\ww$ is $\beta$-H\"older rough at $t$.

If \eqref{eq:TRpvar} holds for a.e. $t$ (and some choice of $(\omega, \beta)$), we say that $\ww$ is a.e. truly rough, resp. $\ww$ is a.e. truly $\beta$-H\"older rough in the case where $\omega(t,s)=|s-t|$ .
\end{definition}

This condition is satisfied a.s. by realisations of (fractional) Brownian motion. More precisely, fBm with Hurst index $H$ is, with probability $1$, a.e. $H$-H\"older rough, see e.g. \cite{FS13}.}

%In this section, we will assume that $\ww$ is a $\alpha$-H\"older rough path, namely
%\[
%\forall 0\leq s \leq t \leq 1, \;\;\; \| \ww_{s,t}\| \leq C|t-s|^{\alpha},
%\]
%and that it is a.e. truly rough in the sense that
%\begin{equation} \label{eq:aeTR}
%\exists \beta < 2\alpha, \mbox{ for a.e. } t \in [0,1), \forall  0\neq \zeta \in \R^d , \quad \limsup_{s \downarrow t} \frac{  \left|w_{t,s} \cdot \zeta \right|}{(s-t)^{\beta}} = +\infty,
%\end{equation}
%where $w$ is the level-$1$ component of $\ww$.

We will also denote by $X:=X^x(\ww)$ the solution to the RDE
\begin{equation*}\label{eq:rde-w}
	dX_t = \sum_i V_i (X_t) d \w^i_t, \;\; X_0=x,
\end{equation*}
and by $J:=J^x(\ww)$ the solution to the RDE 
\begin{equation*}
J_{t\leftarrow t} = Id_{\R^n}, \quad d_u J_{u\leftarrow t} = \sum_{i=1}^d DV_i(X^x_u(\w)) J_{u\leftarrow t} d\w_u^i. 
\end{equation*}

\subsection{Single point control}

\begin{proposition}\label{prop:first-order-condition-point}
Assume that the vector fields $\{V_i\}_{i=1}^d$ are bracket-generating \modif{at $x$} and that $\|\cdot\|_{\cH^\vee}$ is a norm. Let $\ww \in C_g^{p-var}$ be a geometric rough path which is a.e. truly rough \modif{in the sense of Definition \ref{def:TR}}. Then $c_{\ww}(0) > 0$.
\end{proposition}

\begin{proof}
The proof follows from standard arguments (see e.g. \cite[Section 11.3.3]{FH20}) but we give it here for completeness. Fix $\xi$ in $\R^d$, and let us assume that the function
\[
t \in [0,1] \mapsto \left\langle J_{1\leftarrow t} V_i(X_t), \xi \right\rangle
\]
is identically $0$. We aim to show that $\xi=0$. The key observation is that, for any vector field $W$, it holds that
\begin{equation} \label{eq:WBracket}
\forall t \in [0,1], \;\; \left\langle J_{1\leftarrow t} W(X_t), \xi \right\rangle =  \left\langle  W(X_1), \xi \right\rangle  -  \sum_j \int_{t}^1 \left\langle J_{1\leftarrow u} [W,V_j](X_u), \xi \right\rangle d \ww^j_u.
\end{equation}
Applied to $W=V_i$, this implies that $ \sum_j \int_{\cdot}^1 \left\langle J_{1\leftarrow u} [V_i,V_j](X_u), \xi \right\rangle d \ww^j_u $ is identically $0$.

On the other hand, from rough path estimates, it holds that, for any $t \leq s$,
\[
\sum_j \int_{t}^s \left\langle J_{1\leftarrow u} [V_i,V_j](X_u), \xi \right\rangle d \ww^j_u = \sum_{j} \left\langle J_{1\leftarrow t} [\modif{V_i},V_j](X_t), \xi \right\rangle w^j_{s,t} + \modif{ O\left(\| \ww \|_{p-var;[s,t]}^{\frac{2}{p}}\right)}.
\]
Combined with \eqref{eq:TRpvar} this implies that for $i\neq j$, $\left\langle J_{1\leftarrow t} [V_i,V_j](X_t), \xi \right\rangle$ vanishes a.e., and then identically by continuity.

Iterating the above argument, we see that for any $W \in {\rm Lie}(V_1,\ldots, V_n)$, $t \mapsto \left\langle J_{1\leftarrow t} W(X_t), \xi \right\rangle$ vanishes on $[0,1]$. In particular, taking \modif{$t=0$}, we have that $\modif{J_{1 \leftarrow 0}^{T} }\xi$ is orthogonal to all the $W(x)$. But by assumption,  ${\rm Lie}(V_1,\ldots, V_n)(x) = \R^d$ \modif{and since $J_{1 \leftarrow 0}$ is invertible}, this implies that $\xi=0$.
\end{proof}

\modif{
\begin{lemma} \label{lem:TRqvar}
Assume that $\ww \in C^{\alpha}_g$ is a.e. $\beta$-H\"older truly rough for some $\alpha < \beta \leq 2 \alpha$, and $h \in C^{q-var}$ with $ q \geq \frac{1}{\beta}$. Then the Young translation $\ww + h$ $\in$ $C^{\frac{1}{\alpha}-var}_g$ is a.e. truly rough.
\end{lemma}

\begin{proof}

Let $\omega_h(s,t) := \| h\|_{q-var;[s,t]}^q$. By Young translation of rough paths, we may find a constant $C$ such that the function
\[
\omega(s,t) = C |t-s| + C \omega_h(s,t)
\]
satisfies $\omega(s,t) \geq \|\ww + h\|_{1/\alpha-var;[s,t]}^\alpha$. 

We first claim that, for a.e. $t \in [0,1]$,
\begin{equation}\label{eq:hts}
\limsup_{s \downarrow t} \frac{\omega_h(s,t)}{|t-s|}  < \infty.
\end{equation}

Indeed, since $\omega_h$ is a control function, it holds that, for $t \leq s$,
\begin{equation*} 
\omega_h (t,s) \leq \left( \omega(0,s) - \omega(0,t) \right).
\end{equation*}
Let $\mu$ the measure on $[0,1]$ associated to the increasing function $\omega(0,\cdot)$, and write $\mu = f \cdot \lambda + \nu$, where $\lambda$ is Lebesgue measure, $f \in L^1([0,1])$ and $\nu$ is supported on a closed set $A$ s.t. $\lambda(A)=0$. Then, at any $t$ which is a Lebesgue point of $f$ and not in $A$ (which form a full measure set), it holds that, as $s \downarrow t$, $\omega(0,s) - \omega(0,t) = f(t) (t-s) + o((s-t))$, which implies \eqref{eq:hts}.

Now let $t$ be s.t. $\ww$ is $\beta$-H\"older rough at $t$ and \eqref{eq:hts} holds (both properties holds on a set of full measure). It then holds that, for any $0\neq \zeta \in \R^d $,
\begin{align*}
\frac{  \left|w_{t,s} + h_{t,s} \cdot \zeta \right|}{\omega(s,t)^{\beta}}& = \frac{  \left|w_{t,s} \cdot \zeta \right|}{ C'  |t-s|^{\beta}} + O\left( \omega_h(s,t)^{\frac{1}{q}-\beta} \right) \\
& \to +\infty \mbox{ as }s \downarrow t.
\end{align*}

It follows that $\ww+h$ is $(\omega,\beta)$ rough at almost every $t$.
\end{proof}
}

Combining the above results, we immediately obtain that, if $\w$ is truly rough, there is no non-minimal critical point at finite $\cH$-distance from $\w$. (Recalling that Brownian motion is, with probability $1$, a.e. truly \modif{$\frac{1}{2}$-H\"older }rough, this implies Theorem \ref{thm:qualit_intro} (2)).

\begin{corollary}
In the setting of section \ref{sec:defGF}, assume that the vector fields $\{V_i\}_{i=1}^d$ are bracket-generating, that $\|\cdot\|_{\cH^\vee}$ is a norm, that $\ww \in C^\alpha_g$ is a geometric rough path which is \modif{a.e. $\frac{1}{q}$-H\"older rough with $\cH \subset C^{q-var}$}, and that \modif{the cost $g\in C^2(\R^n)$ is such that} $\nabla g = 0 \Rightarrow g=0$ \modif{on $\R^n$}. Then
\[
\forall h \in \cH, \;\; \nabla_{\cH} \cL_{\w}(h) = 0 \Leftrightarrow  \cL_{\w}(h) = 0.
\]
\end{corollary}
\begin{proof}
Immediate from Proposition \ref{prop:first-order-condition-point} and Lemma \ref{lem:TRqvar}.
\end{proof}

As a further corollary, we show that, in a hypoelliptic setting, if the objective function satisfies a \L{}ojasiewicz inequality and the initial condition is truly rough, convergence of the gradient flow trajectory is equivalent to its boundedness (this implies point (3) in Theorem \ref{thm:qualit_intro} from the Introduction).

\begin{corollary}
Assume that the vector fields $\{V_i\}_{i=1}^d$ are bracket-generating and that $\|\cdot\|_{\cH^\vee}$ is a norm, that $\ww \in C^\alpha_g$ is a $\alpha$ geometric rough path which is a.e. truly \modif{$\frac{1}{q}$-H\"older } rough, that $g$ satisfies $| \nabla g | \geq c_g \sqrt{g}$ on $\R^n$ for some $c_g > 0$, and that the embedding $\cH \subset C^{q-var}$ is compact with $\frac 1 q + \alpha > 1$. 

Let $(h(s))_{s \geq 0}$ be the associated gradient flow trajectory defined in section \ref{sec:defGF}. It then holds that
\[
\sup_{s \geq 0} \|h(s)\|_{\cH} < \infty \;\; \Rightarrow \;\; \exists \bar{h} \in \cH \mbox{ s.t. } \lim_{s \to \infty} h(s) = \bar{h}, \;\; \cL_{\w}(\bar{h}) = 0.
\]
\end{corollary}

\begin{proof}
By Proposition \ref{prop:first-order-condition-point} and Lemma \ref{lem:TRqvar}, it holds that the continuous function $C^{q-var}\ni h  \mapsto c_{\w}(h)$ is strictly positive, so that by the compactness of the embedding, for any $r \geq 0$,
\[
\inf_{\| h \|_{\cH} \leq r} \frac{\| \nabla_{\cH} \cL_{\w} (h) \|_{\cH}^2}{ \cL_{\w}(h)} \geq   c_g^2  \left( \inf_{\| h \|_{\cH} \leq r} c_{\w}(h) \right) > 0.
\]
The result then follows from Proposition \ref{cor:GFbdd}.
\end{proof}
 
 \begin{remark}
 The assumption that $| \nabla g | \geq c \sqrt{g}$ globally can be replaced by a local assumption i.e. for any compact $K$, there exists $c_K>0$, $| \nabla g | \geq c_K \sqrt{g}$ on $K$, if we additionally assume that $g$ has compact sublevel sets.
 \end{remark}

\subsection{General measure control}
Let $\mu$ be a compactly supported measure on $\R^n$ (which would correspond to the data in a machine learning setting) and denote its support by $M=Supp(\mu)$. Let $g:\R^n\times \R^n \rightarrow\R_{+}$ be a $C^2$ cost function such that, for $(x,y)\in\R^n\times\R^n$,

\begin{equation}\label{cond:critical-points-of-g}
	\nabla_x g(x,y) = 0 \iff g(x,y)=0,
\end{equation}
where $\nabla_x$ denotes the gradient with respect to the first variable. Given a continuous objective function $y:\R^n\rightarrow\R^n$, we consider the problem of minimising over $\ww \in C^\alpha_g$ the functional 

\begin{equation}\label{eq:loss-mu}
\cL^\mu(\ww) = \int_M g\left( X^{x}_1(\ww), y(x)\right) d\mu(x).
\end{equation}
In this case
\[
\nabla_{\cH} \cL^\mu(\ww) = \int_M \nabla_x g\left( X^{x}_1(\ww),y(x)\right) \cdot_{\R^n} \nabla_{\cH} X^x_1(\ww)\, d\mu(x).
\]

For every $x\in M$, denote by $\xi(x)=\nabla_x g\left( X^{x}_1(\ww),y(x)\right)$. From formula \eqref{eq:formula-eq-gradient}, we get, for  $k\in\mathcal{H}$, 
\begin{align}
    \product{\nabla_{\cH}\cL^\mu(\ww)}{k}_{\cH} &= \int_S  \product{\xi(x) \cdot_{\R^n}\nabla_{\cH} X^x_1(\ww)}{k}_{\cH} d\mu(x)\nonumber\\
    &= \int_S \sum_{i=1}^d \int_0^1  \xi(x) \cdot_{\R^n} J^x_{1 \leftarrow t} V^i(X^x_t(\w))\, dk^i_t\, d\mu(x)\nonumber\\
    &= \sum_{i=1}^d \int_0^1 \int_S \xi(x) \cdot_{\R^n} J^x_{1 \leftarrow t} V^i(X^x_t(\w))\, d\mu(x)\, dk^i_t \label{eq:formula-H-product-under-mu}
\end{align}
where the last equality holds using smooth approximations of $k$ and classic Fubini theorem.

Note that \eqref{eq:loss-mu} can be still interpreted as a control problem over a distribution of points. Control-geometric theory suggests that, in order for the problem to be solvable, stronger conditions on the set of vector fields $\{V_i\}_{i=1}^d$ have to be assumed. For instance, for the result below, we consider the following condition \modif{which is similar to a condition } from \cite{AS22} (recall that $M$ is the support of $\mu$, assumed compact):
\begin{equation}\label{cond:AS22}
\begin{split}
&\mbox{for each continuous function } Y: M\rightarrow \R^n, \mbox{ there holds} \\
&\inf \left\{  \left(\sup_{ x \in M}  |Y(x) - X(x)| \right), \; X \in {\rm{Lie}}(V_1,\ldots, V_d) \right\}= 0.
\end{split}
\end{equation}

\begin{proposition} \label{prop:ncN}
	Assume that the vector fields $\{V_i\}_{i=1}^d$ satisfy condition \eqref{cond:AS22}, that the function $g$ is $C^2(\R^n\times\R^n)$ and satisfies \eqref{cond:critical-points-of-g}, and that $y$ is $C\left(\R^n;\R^n\right)$. Moreover, assume that $\|\cdot\|_{\cH^\vee}$ is a norm. Let $\ww \in C^\alpha_g$ be a geometric rough path \modif{which is a.e. $\beta$-H\"older truly rough for some $\alpha < \beta \leq 2 \alpha$ in the sense of \eqref{def:TR}}. Then 
    \[
    \nabla_{\cH}\cL^\mu(\ww)=0 \iff \cL^\mu(\ww)=0,
    \]
where $\cL^\mu$ is defined in \eqref{eq:loss-mu}.
\end{proposition}

\begin{proof}
    The converse implication is true by construction, let us suppose that $\nabla_{\cH}\cL^\mu(\ww)=0$. In this case, $\product{\nabla_{\cH}\cL^\mu(\ww)}{k}_{\cH}=0$ for every $k\in\cH$ and, since we assume that $\cH^\vee$ is a norm, by \eqref{eq:formula-H-product-under-mu}, it holds
    \begin{equation*}
        \int_M \xi(x) \cdot_{\R^n} J^x_{1 \leftarrow t} V^i(X^x_t(\w))\, d\mu(x)=0\quad\quad\forall i=1,\dots,d.
    \end{equation*}
    Given a smooth vector field $W$ and recalling \eqref{eq:WBracket}, for every $t\in[0,1)$, we have
    \begin{align*}
         \int_M \xi(x) &\cdot_{\R^n} J^x_{1 \leftarrow t} W(X^x_t)\, d\mu(x) \\
         &=  
         \int_M  \xi(x) \cdot_{\R^n} W(X^x_1)   \, d\mu(x) 
         -
         \int_M \sum_j \int_{t}^1 \xi(x)  \cdot_{\R^n}J^x_{1\leftarrow u} [W,V_j](X^x_u) \,   d \ww^j_u \, d\mu(x)
    \end{align*}
    For $W=V^i$, the above left hand side is 0 and therefore, for every $t \leq s$, it holds that
    \begin{equation}\label{eq:vanishing-mu-st-integral}
        \int_M  \sum_j \int_{t}^s  \xi(x)  \cdot_{\R^n}J^x_{1\leftarrow u} [V_i,V_j](X^x_u)  \, d \ww^j_u \, d\mu(x) = 0.
    \end{equation}
    From rough path estimates, it holds that, for any $t \leq s$,
	\begin{equation*}
	\sum_j \int_{t}^s \left\langle J^x_{1\leftarrow u} [W,V_j](X^x_u), \xi(x) \right\rangle d \ww^j_u = \sum_{j} \left\langle J^x_{1\leftarrow t}	[W,V_j](X^x_t), \xi(x)  \right\rangle w^j_{s,t} + R(x,t,s)\modif{\| \ww \|_{p-var;[s,t]}^{\frac{2}{p}}}
	\end{equation*}
	where $R:M\times[0,1]^2\rightarrow\R$ is a continuous bounded function representing the remainder. Plugging the above estimate in \eqref{eq:vanishing-mu-st-integral}, we obtain that
	\[
	 \sum_{j} \int_M\left\langle J^x_{1\leftarrow t}	[W,V_j](X^x_t), \xi(x)  \right\rangle d\mu(x) w^j_{s,t} + \int_SR(x,t,s)\, d\mu(x)\modif{\| \ww \|_{p-var;[s,t]}^{\frac{2}{p}}}=0
	\]
	for every $t \leq s$. Then, as in Proposition \ref{prop:first-order-condition-point}, from true roughness property of $\ww$ we deduce that	 
    \begin{equation*}
        \int_M  \xi(x)  \cdot_{\R^n}J^x_{1\leftarrow u} [V_i,V_j](X^x_u) \, d\mu(x) = 0\quad\quad\forall u\in[0,1).
    \end{equation*}
    Iterating the above argument, we see that for any $W \in {\rm Lie}(V_1,\ldots, V_n)$, 
    \begin{equation*}
        u\mapsto\int_M  \xi(x)  \cdot_{\R^n}J^x_{1\leftarrow u} [W,V_j](X^x_u) \, d\mu(x) 
    \end{equation*}
    vanishes on $[0,1]$ and, consequently, the same is true for 
    \begin{equation*}
        u\mapsto\int_M  \xi(x)  \cdot_{\R^n}J^x_{1\leftarrow u} W(X^x_u) \, d\mu(x).
    \end{equation*}
    Let $u=0$, denote by $J^{x,*}_{1\leftarrow 0}$ the transpose matrix of $J^x_{1\leftarrow 0}$ and let, for every $x\in M$, $\xi^*(x)=J^{x,*}_{1\leftarrow 0}\xi(x)$. Then
    \begin{equation*}
        \int_M  \xi(x)  \cdot_{\R^n}J^{x,*}_{1\leftarrow 0} W(X^x_0) \, d\mu(x) = \int_M  \xi^*(x)  \cdot_{\R^n} W(x) \, d\mu(x) = 0 
    \end{equation*}
    for every $W \in {\rm Lie}(V_1,\ldots, V_n)$. Since \eqref{cond:AS22} holds and $\xi^*$ is a continuous vector field, for every $\varepsilon>0$, there exists $W^\varepsilon\in {\rm Lie}(V_1,\ldots, V_n)$ such that $\sup_{ x \in M}  |W^\varepsilon(x) - \xi^*(x)|<\varepsilon$. Hence, for every $\varepsilon>0$,
    \begin{align*}
        0 = \int_M  \xi^*(x) \cdot_{\R^n} W^\epsilon(x) \, d\mu(x) 
        &=
        \int_M  \xi^*(x) \cdot_{\R^n} \left( W^\epsilon(x) - \xi^*(x) \right)\, d\mu(x) +
        \int_M  |\xi^*(x)|^2 \, d\mu(x)\\
        &\geq
        -\varepsilon \int_M  |\xi^*(x) | \, d\mu(x) +\int_M  |\xi^*(x)|^2 \, d\mu(x)
    \end{align*}
    In the limit for $\varepsilon$ that goes to 0 we obtain that $\int_M  |\xi^*(x)|^2 \, d\mu(x)=0$ and, therefore, $\xi^*(x)=0$ for $\mu\mbox{-a.e.}\;x\in M$. Since $J_{1\leftarrow 0}^{x,*}$ is invertible, it holds that
    \[
    \nabla_x g\left( X^{x}_1(\ww),y(x)\right)=\xi(x)=0 \mbox{  for  } \mu\mbox{-a.e. }x\in M.
    \]
    From the property \eqref{cond:critical-points-of-g} we deduce that $g\left( X^{x}_1(\ww),y(x)\right)=0$ for $\mu$-a.e. $x\in M$. It follows immediately that $\cL^\mu(\ww)=0$.
    
\end{proof}

\begin{remark}
\modif{
Our generalized H\"ormander assumption is assumed to hold for $M$ being the support of the distribution of inputs. In \cite{AS22}, this condition (as well as a slightly stronger one), is assumed for arbitrary compacts $M$, and they show that this implies some form of controllability. (This is similar to the single-point setting, where the H\"ormander condition is only needed at the starting point for non-degeneracy to hold, but the Chow-Rashevskii controllability theorem requires it to hold on the whole space).
}

In case where the measure $\mu$ is finitely supported, the property \eqref{cond:AS22} is implied by the following condition from \cite{CLT20} :
\begin{align*}
&\forall N \geq 1, \;\mbox{ for any distinct }x_1,\ldots, x_N \;\in \R^n, \mbox{ and any } v_1, \ldots, v_N \in \R^n,  \\
&\mbox{ there exists  }V \in  {\rm{Lie}}(V_1,\ldots, V_d) \mbox{ s.t. }V(x_i) = v_i ,\;\;  i=1, \ldots, N.
\end{align*}

Examples of vector fields satisfying this assumption, or the condition \eqref{cond:AS22} for arbitrary compact $M$ are given in \cite{CLT20} and \cite{AS22}.

\modif{
Note that, for general measures $\mu$, the controllability results from \cite{AS22} are only approximate, i.e. we do not expect that $\cL^{\mu}$ attains its minimum value  for a given control, in contrast to the single-point case. This means that the non-degeneracy above is arguably less interesting in this case. (A notable exception is the case where the measure is finitely supported, which can be reduced to the single-point case \cite{CLT20}, so that our Chow-Rashevskii type theorem from section \ref{sec:CR-rough} below does imply that zeroes of $\cL$ exist).
} 
\end{remark}

\section{Chow-Rashevskii with rough drift} \label{sec:CR-rough}

In this section, we prove a version of the classical Chow-Rashevskii theorem with a rough (fixed) drift. The setting will be slightly more general than in previous sections, namely (in this section only) we will not assume that the terms in $\w$ and $h$ are driven by the same vector fields (nor even that they have the same dimension).
 More precisely, consider two families $\{V_i\}_{i=1}^{d}$ and $\{W_j\}_{j=1}^{d'}$ of smooth vector fields on $\R^n$ and fix a geometric rough path $\w$ $\in$ $C^{p-var}_g([0,1];\R^{d})$. Throughout this section, we let $\cH$ be a Hilbert space of $\R^{d'}$-valued paths, which satisfies Complementary Young Regularity, namely $\cH\subset C^{q-var}_g([0,1];\R^{d'})$, where $\frac{1}{p}+\frac{1}{q}>1$. Moreover, assume that the set of $\R^{d'}$-valued smooth functions $C^{\infty}([0,1];\R^{d'})$ is embedded in $\cH$.

Let, for every $x\in\R^n$ and $h\in\mathcal{H}$, $X^x(\w, h)$ be the solution of the following RDE with drift
\begin{equation}\label{eq:rde-with-drift}
    dX_t = \sum_{i=1}^d V_i (X_t) d\w^i_t + \sum_{j=1}^{d'}  W_j (X_t) dh^j_t, \;\; X^x_0=x
\end{equation}
in the sense of Definition 12.1 of \cite{FV10}. The main result of this section is then the following.
\begin{thm}[Chow-Rashevskii with rough drift] \label{thm:chow-rashevskii} 
       \modif{Assume that $\{W_j\}_{j=1}^{d'}$ satisfy the bracket-generating condition.} Then, for all $x,y\in\R^n$ and geometric rough path $\w\in$ $C^\alpha_g$, there exists $h \in \mathcal{H}$ such that $X_1^x(\w, h) = y$. In addition, $h$ may be taken non-singular in the sense that $\left\langle \nabla_{\cH}X^x_1(\w, h), \cdot \right\rangle:\cH\rightarrow\R^n$ is surjective. 
\end{thm}

When $d'=d$, the two families of vector fields coincide, and the cost function $g$ is the squared Euclidean distance with respect to a fixed point $y\in\R^n$, as a particular case of the Theorem \ref{thm:chow-rashevskii}, it holds that the problem of minimising the functional \eqref{eq:def-loss} from section \ref{sec:defGF} is solvable on $\cH$.
\begin{corollary}
	In the setting of section \ref{sec:defGF}, further assume that $C^\infty \subset \mathcal{H}$. Then, for all $x,y\in\R^n$ and geometric rough path $\w\in$ $C^\alpha_g$, there exists $h \in \mathcal{H}$ such that $\cL_{\w}(h)=\left| X_1^x\left( \ww + h\right)-y\right|^2=0$.
\end{corollary}

\begin{remark}
While the main interest of Theorem \ref{thm:chow-rashevskii} for us is in terms of the control problem from section \ref{sec:defGF}, it can also be useful in a probabilistic context. Indeed, consider the mixed rough/stochastic differential equation
\[
    dX_t = \sum_{i=1}^d V_i (X_t) d\w^i_t + \sum_{j=1}^{d'}  W_j (X^x_t) \circ dB^j_t, \;\; X_0=x,
\]
where $\w$ is fixed and $(B_j)_{j=1,\ldots, d'}$ are Brownian motions (or more generally, Gaussian processes with sufficiently rich Cameron-Martin space). 
Such equations are considered \modif{from the point of view of Malliavin calculus } in \cite{BCN24} where they show that, under suitable conditions, for any $t$, the law of $X_t^x$ admits a smooth density $p_t(\cdot)$ with respect to Lebesgue measure.

\modif{For such equations, the ''skeleton equation'', i.e. roughly speaking the one where the stochastic terms are replaced by Cameron-Martin elements, is then exactly of the form \eqref{eq:rde-with-drift} we consider in this section. It is well-known that Malliavin calculus can be used to deduce properties of the law of $X$ from those of the skeleton equation. In particular, controllability results for the latter are linked to strict positivity of the transition densities, as was first proven by Ben Arous and L\'eandre \cite{BL91}. We therefore expect that our Theorem \ref{thm:chow-rashevskii} can be useful in this context, and leave the details to future research.}
% We can then combine Theorem \ref{thm:chow-rashevskii} with a classical criterion from Malliavin calculus due to Ben Arous and L\'eandre \cite{BL91}, to obtain that, under the bracket-generating condition for $W$, $p_t(\cdot)$ is strictly positive on $\R^n$ (we leave the details to the interested reader). 
\end{remark}

\begin{remark}
A problem combining control theory and rough paths is also studied in \cite{Bou23}, where the author proves, in a very general setting, that $\{X_1(\w), \w \mbox{ rough path} \}$ is not bigger than $\{X_1(h), h \mbox{ smooth}\}$. Our problem is different since we have a fixed rough drift which we cannot remove (but we work in a more restricted setting since we assume a bracket-generating condition). \end{remark}

\subsection{Density of endpoints}

\begin{lemma}\label{lemma:dense-image}
    Assume that the vector fields $\{W_j\}_{j=1}^{d'}$ are bracket-generating. Given a geometric rough path $\w$ $\in$ $C^{p-var}_g$, the set of endpoints 
    $E^x_1(\w) = \{ X^x_1(\w, h)| h\in\mathcal{H} \}$ is dense in $\R^n$.
\end{lemma}
\begin{proof}

Given a geometric rough path $\w\in C^{p-var} \left([T_0,T_1], G^{\lfloor p \rfloor}(\R^d) \right)$ and a continuous monotone surjective function $f:[T'_0,T'_1]\rightarrow[T_0,T_1]$, we denote by $\w\circ f$ the geometric rough path in $C^{p-var} \left([T'_0,T'_1], G^{\lfloor p \rfloor}(\R^d) \right)$ such that
\[
(\w\circ f)_t=\w_{f(t)},\quad \forall t\in [T'_0,T'_1].
\]

    Fix $\bar{y}\in\R^d$ and denote by $z=X^x_1(\w,0_\cH)$. By classic Chow-Rashevskii Theorem (see e.g. \cite{Rif14}), there exists $h\in\mathcal{H}$ such that $h_0=0$ and $X^z_1(0_{C^{p-var}_g},h)=\bar{y}$. Our goal is constructing a sequences $\{h^k\}_{k\geq 2}\subset\mathcal{H}$ such that, for large $k$, $X^x_1(\w, h_k)$ is arbitrarily close to $\bar{y}$.
    
    For $k\geq2$, define $\varphi_k:[0,1]\rightarrow[0,1-\frac{1}{k}]$ and $\psi_k:[0,1]\rightarrow[1-\frac{1}{k},1]$ such that, for every $r\in[0,1]$,
    \[
    \varphi_k(r)=\frac{k-1}{k}r \;\;\mbox{and}\;\; \psi_k(r)=\frac{r}{k}+1-\frac{1}{k}.
    \]
    Denote by $\tilde{h}^k:[0,1]\rightarrow\R^d$ the path such that $\tilde{h}^k_r=(h\circ\psi^{-1}_k)_r$ for every $r\in[1-\frac{1}{k},1]$ and is equal to zero otherwise. Note that $\tilde{h}^k\in\mathcal{H}$. We claim that the sequence defined by $\tilde{h}_k$, for $k\geq 2$, follows our requirements. 
    
    Denote by $\Delta=\{(s,t):0\leq s \leq t \leq 1\}$ and let $\omega: \Delta \rightarrow [0,+\infty)$ such that
    \begin{equation*}
        \omega(s,t) = \|\ww\|^p_{p-var;[s,t]},\quad\forall (s,t)\in\Delta.
    \end{equation*}
    The function $\omega$ is the so called \textit{control} of $\ww$ in the terminology of \cite{FV10}. It is well-known that $\omega$ is continuous and, for every $s\in[0,1]$, $\omega(s,s)=0$. In particular, being defined on a compact set, $\omega$ is uniformly continuous. 
    
    We claim now that the sequence $\{\w\circ\varphi_k\}_{k\geq2}$ is \textit{equibounded, equicontinuous} and \textit{equibounded in} $p$-\textit{variation norm}. Indeed, for $k\geq2$ and $(s,t)\in\Delta$, we have the uniform bound
    \begin{equation*}
        |(\w\circ\varphi_k)_t| = |\w_{\varphi_k(t)}| 
        \leq 
        \sup_{t\in[0,1]}|\w_t| < \infty
    \end{equation*}
    and
    \begin{equation*}
        d\left((\w\circ\varphi_k)_s,(\w\circ\varphi_k)_t\right)
        =
        d\left(\w_{\varphi_k(s)},\w_{\varphi_k(t)}\right)
        \leq
        \|\w\|_{p-var;[\varphi_k(s),\varphi_k(t)]}
        = \omega(\varphi_k(s),\varphi_k(t))^{1/p},
    \end{equation*}
    where
    \begin{equation*}
        |\varphi_k(t)-\varphi_k(s)|=\frac{k-1}{k}|t-s|\leq|t-s|.
    \end{equation*}
    Therefore, the uniform continuity of $\omega$ implies the (uniform) equicontinuity of the sequence. 
    Moreover, by invariance of the $p$-$var$ norm under reparametrisation, we have the following uniform bound 
    \begin{equation*}
        \|\w\circ\varphi_k\|_{p-var;[0,1]} = \|\w\|_{p-var;[0,1-1/k]} 
        \leq
        \|\w\|_{p-var;[0,1]}, \quad \forall k\geq 2.
    \end{equation*}
    Fix $p'>p$. Proposition 8.17 of \cite{FV10} then assures that the sequence $\{\w\circ\varphi_k\}_{k\geq2}$ is compact in $C^{p'-var}_g$.
    Repeating the same arguments one obtains that also the sequence $\{\w\circ\psi_k\}_{k\geq2}$ is compact in $C^{p'-var}_g$.

    Finally, we prove that the two sequences converge uniformly to explicit limits. Let $\w^1$ denote the constant element of $C^{p'-var}$ such that $\w^1_t=\w_1$ for every $t\in[0,1]$. For $k\geq2$ and $t\in[0,1]$, we have
    \begin{equation*}
        d\left((\w\circ\varphi_k)_t,\w_t\right)
        =
        d\left(\w_{\varphi_k(t)},\w_t\right)
        =
        \|\w\|_{p-var;[\varphi_k(t),t]}
        = \omega(\varphi_k(t),t)^{1/p}
    \end{equation*}
    and
    \begin{equation*}
        d\left((\w\circ\psi_k)_t,\w^1_t\right)
        =
        d\left(\w_{\psi_k(t)},\w_1\right)
        =
        \|\w\|_{p-var;[\psi_k(t),1]}
        = \omega(\psi_k(t),1)^{1/p},
    \end{equation*}
    where 
    \begin{equation*}
        |t-\varphi_k(t)|=\frac{1}{k}t\leq\frac{1}{k}\quad\text{and}\quad|1-\psi_k(t)|=\frac{1-t}{k}\leq\frac{1}{k}.
    \end{equation*}
    The uniform convergence of $\{\w\circ\varphi_k\}_{k\geq2}$ and $\{\w\circ\psi_k\}_{k\geq2}$ towards, respectively, $\w$ and $\w^1$ follows from the uniform continuity of $\omega$. This implies the unicity of the respective accumulation points in $C^{p'-var}_g$ for the two sequences and, by sub subsequence lemma, we infer that $\{\w\circ\varphi_k\}_{k\geq2}$ and $\{\w\circ\psi_k\}_{k\geq2}$ converge towards, respectively, $\w$ and $\w^1$ in $C^{p'-var}_g$.

    Since $\w^1$ is constant, notice that $\|\w^1\|_{p-var;[0,1]}=0$ and $\bar{y}=X^z_1(\w^1,h)$. Define, for $k\geq 2$, $z_k = X^x_1(\w\circ\varphi_k)$ and $y_k=X^{z_k}_1(\w\circ\psi_k, h)$. By stability estimates for RDEs (Theorem 12.11 of \cite{FV10}), we have
    \[
       |z_k-z| \leq C {\rho_{p'-var}}(\w\circ\varphi_k,\w), 
    \]
    and
    \[
       |y_k-\bar{y}| 
       \leq 
       C'\left( |z_k-z| + {\rho_{p'-var}}(\w\circ\psi_k, \w^1) \right)
       \leq
       C''\left( \rho_{p'-var}(\w\circ\varphi_k,\w) + {\rho_{p'-var}}(\w\circ\psi_k, \w^1) \right). 
    \]
    Since $\w\circ\varphi_k$ and $\w\circ\psi_k$ converge, respectively, to $\w$ and $\w^1$ in $C^{p'-var}\left([0,1];G^N(\R^d)\right)$, the $y_k$ can get arbitrarily close to $\bar{y}$ as $k$ goes to infinity.  
\end{proof}

\subsection{Non-empty interior of endpoints}

Denote, for every $h\in\mathcal{H}$, by
\[
Im^x_1(h)=\bigl\{\product{\nabla_{\mathcal{H}}X^x_1(\w, h)}{k}_{\mathcal{H}}\big\vert k\in\mathcal{H}\bigr\}
\]
and notice that $Im^x_1(h)\subseteq\R^n$ is a linear subspace.

The arguments used to prove the following two lemmas are adapted from Rifford (\cite{Rif14}) to suit our (rough) setting.

\begin{lemma}\label{lemma:value-vector-fields-in-the-image-of-the-gradient}
    For every $h\in\mathcal{H}$ and for every $j=1,\dots,d'$,
    \[
    W_j \left( X^x_1(\w,h) \right) \in Im^x_1(h).
    \]
\end{lemma}
\begin{proof}
    Recall from \eqref{eq:formula-eq-gradient} that, for every $k\in\mathcal{H}$,
    \begin{equation*} 
        \product{\nabla_{\mathcal{H}}X^x_1(\w, h)}{k}_{\mathcal{H}}
        =
        \sum_{i=1}^d \int_0^1  J_{1 \leftarrow t} W_j(X^x_t(\w,h)) dk^i_t,
    \end{equation*}
    where the Jacobian $J_{1 \leftarrow t}$ solves a RDE of the form \eqref{eq:formula-eq-jacobian}.

    Fix $j\in\{1,\dots,d'\}$, and denote by $e_j$ the $j$-th element of the canonical basis in $\R^d$. For every $\varepsilon\in(0,1)$, let $f^{\varepsilon}:[0,1]\rightarrow\R$ be a smooth, non-decreasing function such that
    \[
    f^{\varepsilon} \equiv 0 \mbox{ on }[0,1-\varepsilon]\quad\mbox{and}\quad f^{\varepsilon} \equiv 1 \mbox{ on }[1-\varepsilon+\varepsilon^2,1],
    \]
    and define the smooth path $k^\varepsilon\in\mathcal{H}$ such that $k^{\varepsilon}(0)=0$ and its gradient $\dot{k}^\varepsilon$ is given by
    \begin{equation*}
        \dot{k}^\varepsilon_t = 
        \begin{cases}
            0 & \mbox{if } 0\leq t \leq 1-\varepsilon\\
            (1/\varepsilon)f_t^\varepsilon\, e_j & \mbox{if } 1-\varepsilon < t \leq 1
        \end{cases}.
    \end{equation*}
    We have
    \begin{equation*}
        \product{\nabla_{\mathcal{H}}X^x_1(\w, h)}{k^\varepsilon}_{\mathcal{H}}
        =
        \frac{1}{\varepsilon} \int_{1-\varepsilon}^1  J_{1 \leftarrow t} W^j(X^x_t(\w,h)) f^\varepsilon_t dt.
    \end{equation*}
    Hence
    \begin{align*}
        \big| &\product{\nabla_{\mathcal{H}}X^x_1(\w, h)}{k^\varepsilon}_{\mathcal{H}} - W_j(X^x_1(\w,h)) \big| \\
        &\quad =\left| \frac{1}{\varepsilon} \int_{1-\varepsilon}^1  J_{1 \leftarrow t} W_j(X^x_t(\w,h)) f^\varepsilon_t dt - \frac{1}{\varepsilon} \int_{1-\varepsilon}^1 W_j(X^x_1(\w,h)) dt\right|\\
        &\quad \leq \frac{1}{\varepsilon} \int_{1-\varepsilon}^1 \left| J_{1 \leftarrow t} W_j(X^x_t(\w,h))f^\varepsilon_t  - W_j(X^x_1(\w,h)) \right| dt\\
        &\quad \leq \frac{1}{\varepsilon} \int_{1-\varepsilon}^1 \Big(\|J_{1 \leftarrow t}-Id_{\R^n}\|_{\modif{F}} \left| W_j(X^x_t(\w,h))\right| f^\varepsilon_t 
        + \left| W_j(X^x_t(\w,h)) - W_j(X^x_1(\w,h)) \right| f^\varepsilon_t \\ 
        &\quad\quad\quad\quad\quad\quad\quad\quad\quad\quad\quad\quad\quad\quad\quad\quad\quad\quad\quad\quad\quad\quad\quad\quad\quad\quad\quad\quad\quad+ \left| W_j(X^x_1(\w,h)) \right| |f^\varepsilon_t - 1| \Big)  dt\\
        &\quad \leq \frac{1}{\varepsilon} \int_{1-\varepsilon}^1 \Big(\|J_{1 \leftarrow t}-Id_{\R^n}\|_{\modif{F}} \left| W_j(X^x_t(\w,h))\right| 
        + \left| W_j(X^x_t(\w,h)) - W_j(X^x_1(\w,h)) \right| \Big)  dt \\ 
        &\quad\quad\quad\quad\quad\quad\quad\quad\quad\quad\quad\quad\quad\quad\quad\quad\quad\quad\quad\quad\quad\quad\quad\quad\quad\quad\quad\quad\quad+\varepsilon \left| W_j(X^x_1(\w,h)) \right|,
     \end{align*}
     \modif{where $\|\cdot\|_F$ is the Frobenius norm of an $n\times n$ matrix.} By regularity properties of RDEs, both mappings $t\mapsto J_{1 \leftarrow t}$ and $t\mapsto X^x_t(\w,h)$ are continuous and $J_{1 \leftarrow 1} = Id_{\R^n}$. Therefore, it holds
    \[
    \lim_{\varepsilon\rightarrow 0^+}\product{\nabla_{\mathcal{H}}X^x_1(\w, h)}{k^\varepsilon}_{\mathcal{H}} = W_j(X^x_1(\w,h)). 
    \]
    Since $Im^x_1(h)$ is a linear subspace of $\R^n$, it is closed.  We conclude that $W_j(X^x_1(\w,h))\in Im^x_1(h)$. 
\end{proof}

\begin{lemma}\label{lemma:non-empty-intern}
	Assume that the vector fields $\{W_j\}_{j=1}^{d'}$ are bracket-generating.
    Given a geometric rough path $\w$ $\in$ $C^\alpha_g$, let, for every $h\in\mathcal{H}$, $X(\w, h)$ be the solution of \eqref{eq:rde-with-drift}. Then, the set of endpoints 
    $E^x_1(\w) = \{ X^x_1(\w, h)| h\in\mathcal{H} \}$ has non-empty interior. In particular, there exists an open subset $U$ of $E^x_1(\w)$ attained by non-singular controls.
\end{lemma}
\begin{proof}
    Let
    \[
        m = \max\left\{ \text{dim}\left(Im^x_1(h)\right) | h\in\mathcal{H} \right\}
    \]
    Since the set of vector fields $\{W_j\}_{j=1}^{d'}$ satisfies the H\"ormander condition, $\text{Span}\{W_1(x'),\dots, W_{d'}(x')\}$ has dimension at least 1 for every $x'\in \R^n$. 
    From Lemma \ref{lemma:value-vector-fields-in-the-image-of-the-gradient} we have that, for every $1\leq j\leq d'$ and every $h\in\mathcal{H}$, $W_j\big(X^x_1(\w, h)\big)\in Im^x_1(h)$ and therefore $m\geq 1$.
    Hence, there exist $h^*,h^1,\dots,h^m\in\mathcal{H}$ such that $\text{dim}\left(Im^x_1(h^*)\right)=m$ and the linear map
    \begin{align*}
        \mathcal{L}:\R^m &\rightarrow \R^n\\
        \l = (\l^1,\dots,\l^m) &\mapsto \product{\nabla_{\mathcal{H}}X^x_1(\w, h^*)}{\sum^m_{j=1}\l^j h^j}_{\mathcal{H}}
    \end{align*}
    is injective. 
    Since the mapping $h\mapsto\text{dim}( Im^x_1(h) )$ is lower semicontinuous, the rank of a control $h$ is equal to $m$ as soon as $h$ is close enough to $h^*$ in $\mathcal{H}$. 
    Hence, by the rank theorem (see e.g. \cite{lee22}), there exists an open neighbourhood $\mathcal{V}$ of $0\in\R^m$ where the mapping
    \begin{align*}
        \mathcal{E}:\mathcal{V} &\rightarrow \R^n\\
        \l = (\l^1,\dots,\l^m) &\mapsto X^x_1\left(\w,h^*+\sum^m_{j=1}\l^j h^j\right)
    \end{align*}
    is an embedding and whose image is a submanifold $S$ of class $C^1$ in $\R^n$ of dimension $m$. Moreover, by construction, there holds, for every $\l\in\mathcal{V}$,
    \[
    Im^x_1\left(h^*+\sum^m_{j=1}\l^j h^j\right) = D_{\l}\mathcal{E}(\R^m)=T_{\mathcal{E}(\l)}S. 
    \]
    By Lemma \ref{lemma:value-vector-fields-in-the-image-of-the-gradient}, we infer that
    \[
    W_j(\mathcal{E}(\l)) = W_j\left(X^x_1\left(\w,h^*+\sum^m_{j=1}\l^j h^j\right)\right) \in T_{\mathcal{E}(\l)} S,\quad\forall j = 1,\dots,d',\;\forall l\in \mathcal{V},
    \]
    or, equivalently,
    \[
    W_j(x') \in T_{x'} S\quad\forall j = 1,\dots,d',\;\forall x'\in S.
    \]
    It is well-known in differential geometry (see Lemma 1.13 in \cite{Rif14}) that the above line implies that also all the Lie brackets involving $W_1(x'),\dots,W_{d'}(x')$ belong to $T_{x'}S$, for all $x'\in S$. Since the vector fields satisfy H\"ormander condition, $T_{x'}S$ must be of dimension $n$ for every $x'\in S$, from which we deduce that $m=n$. In particular, $h^*$ is non-singular and the same is true for all elements of the set $\{h^*+\sum^m_{j=1}\l^j h^j\vert\, (\l^1,\dots,\l^m)\in\mathcal{V}\}$. 
    Finally, $U:=\mathcal{E}(\mathcal{V}) \subset E^x_1(\w)$ is a non-empty open subset which concludes the proof.    
\end{proof}

\begin{remark}
    Lemmas \ref{lemma:dense-image} and \ref{lemma:non-empty-intern} hold true on arbitrary time interval, that is, for every $0<t\leq1$, $E^x_t(\w) = \{ X^x_t(\w, h)| h\in\mathcal{H}\}$ is dense and has non-empty interior.
\end{remark}

\begin{proof}[Proof of Theorem \ref{thm:chow-rashevskii}]
    Consider $\w^1,\w^2:[0,1/2]\rightarrow \R^d$ such that, for every $t\in [0,1/2]$,
    \[
    \w^1_t=\w_t \mbox{ and } \w^2_t= \w_{1-t}.
    \]
    We know from Lemma \ref{lemma:non-empty-intern} that the set $E^x_{1/2}(\w^1) = \{ X^x_{1/2}(\w^1, h)| h\in\mathcal{H} \}$ contains an open subset attained by non-singular controls, and, from Lemma \ref{lemma:dense-image}, that $E^y_{1/2}(\w^2)$ is dense in $\R^n$. It follows that the intersection $E^x_{1/2}(\w^1)\cap E^y_{1/2}(\w^2)$ is non-empty and, hence, there exist $h^1,h^2:[0,1/2]\rightarrow \R^{d'}$ in $\cH\vert_{[0,1/2]}$, with $h^1$ non-singular, such that
    \[
        X^x_{1/2}(\w^1, h^1) = X^y_{1/2}(\w^2, h^2).
    \]
    Moreover, we can assume that $h^1_{1/2}=h^2_{1/2}$. Define $h:[0,1]\rightarrow \R^{d'}$ such that
    \begin{equation*}
    	h_t = \begin{cases}
        h^1_t & \text{if } t \in [0,1/2]\\
        h^2_{1-t} & \text{if } t \in [1/2,1]
    \end{cases},
    \end{equation*}
    and observe that $h\in\mathcal{H}$. By construction, it holds that $X^x_1(\w, h)=y$.
    
    It remains to prove that $h$ is non-singular. For every $k\in\cH$, define $\tilde{k}\in\cH$ such that
    \[
    	\tilde{k}\equiv k\text{ on }[0,1/2],\quad\text{and}\quad\tilde{k}\equiv k_{1/2}\text{ on }[1/2,1].
    \]
    Notice that, by construction of $h$ and $\tilde{k}$,
    \[
    	\product{\nabla_{\cH}X^x_1(\w, h)}{\tilde{k}}_{\cH} = J_{1\leftarrow 1/2}\product{\nabla_{\cH}X^x_{1/2}(\w^1, h^1)}{k\vert_{[0,1/2]}}_{\cH\vert_{[0,1/2]}}.
    \]
    Since the Jacobian $J_{1\leftarrow 1/2}$ is invertible and $h^1$ is a non-singular control, it follows that $Im^x_1(h)=\R^n$ and this concludes the proof. 
\end{proof}

\section{Global convergence results} \label{sec:conv}

In this section we will obtain global convergence results for the gradient flow, i.e., when started from (almost) any initial condition. We will assume that $\cH$ is a Sobolev space $H_0^{\delta}$, for some $\delta \in (1/2,1]$, and we first record some facts on this space, as well as the dual norm $\|\cdot\|_{\cH^{\vee}}$ defined in \eqref{eq:defHvee}. % is equivalent to the $H^{1-\delta}$ norm, which we record in the lemma below.

\begin{lemma} \label{lem:Hdeltavee}
Let $\delta \in (1/2,1]$, then it holds that :

(i) Let $q = \delta^{-1}$, then  $H_0^{\delta} \subset C^{q-var}$ and in fact, there exists $C_1$ s.t., for any $u<v \in [0,1]$, $\| f\|_{q-var;[u,v]} \leq C_1 \| f \|_{H_0^\delta} (v-u)^{\delta - 1/2}$.

(ii) There exists $C_2>0$ s.t. for all $f$, $\| f \|_{(H_0^{\delta})^{\vee}} \geq C_2 \|f \|_{H^{1-\delta}}$.

(iii) There exists $C_3>0$ s.t., for any $0< \eta <1$, any $f \in(H_0^{\delta})^{\vee}$, it holds that
\begin{equation} \label{eq:lowerboundHvee}
\left\| f \right\|_{(H_0^{\delta})^{\vee}} \geq C_3  \eta^{ \delta - \frac{1}{2}}\sup_{i=1,\ldots, d}\, \inf_{t \in [1-\eta, 1]} |f_i(t)| . 
\end{equation}
\end{lemma}

\begin{proof}
(i) is proven in \cite{FV06}.

The proof of (ii) is an immediate consequence of the Fourier characterisation of Sobolev norms, and is deferred to subsection \ref{subsec:Fourier} below.

For (iii) : let $\phi_{\eta}(s) = \left( 1 -  \eta^{-1}(1-s) \right)_+$, an immediate computation gives $\|\phi_\eta\|_{H^{\delta}} = C \eta^{1/2-\delta}$, so that
\[ \left\| f \right\|_{(H_0^{\delta})^{\vee}} \geq C \eta^{\delta-1/2} \sup_i \left|  \eta^{-1} \int_{1-\eta}^1 f^i(s) ds \right| \geq \eta^{\delta-1/2} \sup_{i=1,\ldots, d} \inf_{[1-\eta, 1]} |f_i|.
\]
\end{proof}

We will assume that the objective function $g: \R^n \to \R$ satisfies a local \L{}ojasiewicz inequality and has compact sublevel sets :

\begin{equation} \label{eq:locLoj}
\forall \mbox{ compact } K \subset \R^n, \; \exists c_K>0, \quad | \nabla g | \geq c_K \sqrt{g}\mbox{ on } K.
\end{equation}
\begin{equation} \label{eq:CompSub}
\forall r \geq 0, \quad \left\{ x \; : g(x) \leq r \right\} \mbox{ is compact}.
\end{equation}

\subsection{Elliptic case}

In this section we will prove convergence for arbitrary initial condition in an elliptic setting, i.e., assuming that
\begin{equation} \label{eq:elliptic}
\forall x \in \R^n,  \exists c >0\text{ s.t.} \,\forall q \in \R^n, \quad \sum_{i} \left\langle V_i(x), q \right \rangle^2 \geq c |q |^2.
\end{equation}

\begin{thm} \label{thm:elliptic}
Let $\mathcal{H} = H^{\delta}_0$ for some $\delta \in (1/2, 1]$, and $\w \in \mathcal{C}_g^{p-var}$, with  $\frac{1}{p} + \delta > 1$.
Assume that $g\in C^2(\R^n)$ satisfies \eqref{eq:locLoj}-\eqref{eq:CompSub}, and that the $V_i$ are $C^{\infty}_b$ and such that \eqref{eq:elliptic} above holds.
Then, for the gradient flow $(h(s)) _{s \geq 0}$ defined in section \ref{sec:defGF}, it holds that
\[ \exists \bar{h} \in \cH, \;\;\; \lim_{s \to \infty} h(s) = \bar{h} \mbox{ in } \mathcal{H}, \mbox{ with } \mathcal{L}(\bar{h}) = 0. \]
\end{thm}

\begin{proof}
Note that by Lemma \ref{lem:Hdeltavee} (i), the assumption on $p$ and $\delta$ ensures $H_0^{\delta} \subset C^{q-var}$ with $q=\delta^{-1}$ satisfying $\frac{1}{q} + \frac{1}{p} > 1$ (complementary Young regularity), so that the gradient flow is well-defined. 

By Corollary \ref{cor:c1r} and Proposition \ref{prop:LojL},  it is enough to prove that, for any $\alpha \geq 0$, there exists a ($\w$-dependent) constant $c>0$ s.t.
\begin{equation} \label{eq:cr}
\forall r \geq 0, \quad \inf\left\{  c_{\w}(h) \,:\, \|h\|_{\cH} \leq r, \,\cL_{\w}(h) \leq \alpha \right\} \geq \frac{c}{1+r}.
\end{equation}

We now fix $\alpha \geq 0$, and consider $h \in \cH$ with $\cL_{\w}(h) \leq \alpha$. 

Recall that $c_{\w}(h) =  \inf_{|\xi|=1} \| \phi_{\xi} \|_{(H_0^\delta)^\vee}$, where $\phi_\xi(t) = \left( \left \langle J_{1 \leftarrow t} V^i(X_t), \xi\right \rangle \right)_{i=1,\ldots,d}$, where $J$, $X$ correspond to $\w+h$. 

Let $\xi$  be a vector in $\R^n$ with $|\xi|=1$. Let $z = X_1^x(\w + h)$. Since $g$ has compact sublevel sets, the ellipticity constant is uniform over all $z$ with $g(z) \leq \alpha$, and in fact there exists $c_0>0$ (not depending on $h$ or $\xi$) such that, there exists $i$ with $|V_i(z) \cdot \xi  | \geq c_0$. % Assume $V_i(z) \cdot \xi \geq c_0$ (the other case is treated similarly.

Recall that $(X_t,J_{1 \leftarrow t})$ is the solution to a RDE driven by $\w+h$, with value $(z,I)$ at $t=1$. By rough path bounds and Young translation bounds, there exists a ($\w$-dependent) $c_1>0$ and $t_0 \in [0,1]$ s.t.,
\[
t_0 \leq t \mbox{ and }\|h\|_{q-var;[t,1]} \leq c_1 \;\;\Rightarrow \;\;  \left|\left \langle J_{1 \leftarrow u} V^i(X_u), \xi\right \rangle \right| \geq \frac{c_0}{2}, \; \forall u \in [t,1].
\]
By Lemma \ref{lem:Hdeltavee} (i), the l.h.s. holds if $t$ is chosen such that
\[
(1-t)^{\delta - 1/2}  =  (1-t_0)^{\delta - 1 /2} \wedge  \frac{c_1}{C_2 \| h \|_{\cH}},
\]
while by Lemma \ref{lem:Hdeltavee} (iii), the r.h.s. in the implication above further implies
\[
\left\| \phi_\xi \right\|_{(H_0^\delta)^\vee} \geq \frac{C_3 c_0}{2} (1-t)^{\frac{1}{2} - \delta} \geq C  \wedge \frac{C'}{\|h\|_{\cH}},
\]
for some $C,C'$ which do not depend on $h$ and $\xi$. This concludes the proof of \eqref{eq:cr}.
\end{proof}

\begin{remark} \label{rem:SC1}
Convergence results for our gradient flow can also be obtained from the results of Sussmann and Chitour \cite{Sus93,CS98,Chi06}. They work under the so-called Strong Bracket Generating condition :
\[
\forall x \in \R^n , \, \forall \theta = (\theta^1, \ldots, \theta^d) \in \R^d,\quad {\rm{span}} \{ V^i(x), [ \theta \cdot V , V^i ](x), i=1,\ldots, d \} = \R^n
\]
(which is of course weaker than ellipticity). Under this assumption, all the nonconstant paths are nonsingular (as first observed by \cite{Bis84}) and in fact, it was proven in the above mentioned works that for any compact $K$ disjoint from $\{x\}$ (the starting point) there exists $C_K >0$ such that
\[
\inf  \{ c_0(h)  \, : \, X^{x}(h) \in K, \|h \|_{H^1} \leq r \} \geq \frac{C_K}{1+r}.
\] 
This implies that, for the gradient flow with $\cH=H^1$, starting from an initial condition $h^0 \in H^1$, either $h(s)$ converges to a minimizer or there exists a subsequence $s_n \to \infty$ such that $X^x(h(s_n)) \to x$. The second possibility can be ruled out if, for instance, $g(X^x(h_0)) \leq g(x)$. (It would also be possible to replace $H^1$ by $H^k$ with $k \geq 1$, see \cite{Chi06}).  

Our result above is weaker since we need to assume ellipticity, but on the other hand we can treat arbitrary rough initial conditions. It is possible that Sussmann and Chitour's results extend to this case, but it is not obvious, since their proof crucially relies on inequalities of the form $|\int \phi(t) dh(t)| \lesssim (\sup |\phi(t)|) \|\dot{h} \|_{L^2}$.
\end{remark}

\begin{remark}
Recall from section \ref{subsec:motiv} that the $N$-point control problem, relevant in machine learning, consists, being given distinct initial data $x_1,\dots, x_N$ $ \in$ $\R^n$ and distinct targets  $y_1,\ldots, y_N$, in finding a control $h$ such that
\[ \forall i = 1,\ldots, N,  \quad X^{x_i}_1(h) = y_i. \]
This can be recast as a one-point control problem, where the state space is $\Sigma = (\R^n)^N\setminus \{ (x_i)_{i=1,\ldots, N} : \exists i \neq j, x_i=x_j \}$ (which is connected if $n \geq 2$). In this context, the ellipticity condition states that
\[
\forall \mbox{ distinct } x_1,\ldots, x_N, \quad span \left\{ \begin{pmatrix} V_i(x_1) \\ \ldots \\ V_i(x_N) \end{pmatrix} ; i=1, \ldots, d \right\} = \R^{nN}.
\]
Such vector fields can be proven to exist (but note that this requires at least $d \geq nN$).

Unfortunately, our convergence analysis cannot cover the case of quadratic loss, i.e.
\[
g(x_1, \ldots, x_N) = \sum_{i} \left| x_i - y_i \right|^2
\]
because it does not have compact sublevel sets in $\Sigma$ (and the ellipticity constants deteriorate as $|x_i - x_j| \to 0$). This means that we cannot rule out that the trajectories of the points $X^{x_i}(h(s))$ converge to a point with two identical coordinates (instead of the target $(y_i)$). (See \cite{BPV22} for a way to overcome this diffculty by embedding into a higher dimensional space).
\end{remark}
%\vspace{5mm}
%
%Note however that this very close (any higher power of $\|h\|$ would not work) which means that this method will break down in the hypoelliptic case...

%Q (maybe much later): Can we improve the above analysis ?? What about if $H^1$ is replaced by another function space ??

\subsection{Step-$2$ nilpotent case: a.s. convergence}

In this section, we obtain an almost sure convergence result for the gradient flow started at a Brownian motion. We will assume that our family of vector fields is step-$2$ nilpotent (and bracket generating), namely that
\begin{equation} \label{eq:nilp}
\forall x \in \R^n,  \forall i,j,k \in \{1,\ldots, d\}, \quad [[V_i,,V_j],V_k](x) = 0,
\end{equation}
\begin{equation} \label{eq:step2}
\forall x \in \R^n, \exists c >0, \forall \xi \in \R^n, \quad \sum_{1\leq i \leq d} \product{V_i(x)}{\xi}^2 + \sum_{1 \leq i<  j \leq d} \product{[V_i,V_j](x)}{\xi}^2 \geq c |\xi|^2
\end{equation}

Our main result is as follows.

\begin{thm} \label{thm:mainstep2}
Let $\mathcal{H} = H^{\delta}_0$ for some $\delta \in (1/2, 1]$. Assume that $g\in C^2(\R^n)$ satisfies \eqref{eq:locLoj}-\eqref{eq:CompSub}, and that the $V_i$ are $C^{\infty}_b$ and satisfy \eqref{eq:nilp}-\eqref{eq:step2}.

Let $B=(B^1,\ldots, B^d)$ be a $d$-dimensional Brownian motion. Then, almost surely, for the gradient flow $(h(s)) _{s \geq 0}$ defined in section \ref{sec:defGF} starting at $\w = \mathbf{B}(\omega)$, it holds that
\[ \exists \bar{h} \in \cH, \quad \lim_{s \to \infty} h(s) = \bar{h} \mbox{ in } \mathcal{H}, \mbox{ with } \mathcal{L}(\bar{h}) = 0. \]
\end{thm}

The proof of the Theorem will rely on the following proposition, the proof of which is deferred to subsection \ref{subsec:Fourier} below.
\begin{proposition} \label{prop:BH1L2}
Let $1/2 < \delta \leq 1$ and let $B^1, \ldots, B^d$ be independent Brownian motions. Then, almost surely, there exists $C > 0$ s.t.

\begin{equation} \label{eq:BH1L2}
\forall  h  \in H^{\delta}, \quad  \inf_{c \in \R} \inf_{\sum_{i=1}^d (\lambda^i)^2 =1} \ \left\|  \sum_{i=1}^d \lambda_i B^i - h  - c\right\|_{H^{1-\delta}} \geq  C \frac{ 1 }{ 1 + \| h \|_{ H^\delta}}. 
\end{equation}
\end{proposition}

\begin{proof}[Proof of Theorem \ref{thm:mainstep2}]
We fix $p > 2$ s.t. $\frac{1}{p} + \delta > 1$ and recall that a.s., Brownian motion can be lifted to a  geometric rough path $\w=\mathbf{B}(\omega)$ in $C_{g}^{p-var}$.

Let $(X_t, J_{1 \leftarrow t})$ be the solution to \eqref{eq:rde-h}-\eqref{eq:formula-eq-jacobian} driven by $\w+h$, and for $\xi$ in the unit sphere of $\R^n$ let $\phi_{\xi}$ in $C^{p-var}$ be defined by
\[
\phi_{\xi}(t) = \left( \left\langle J_{1 \leftarrow t} V_i (X_t),  \xi  \right\rangle \right)_{i=1,\ldots,d}.
\]
As in the proof of Theorem \ref{thm:elliptic}, it will be enough to prove
\[
\left\|  \phi_{\xi} \right\|_{ (\cH)^{\vee}} \geq  \frac{c}{1 + \| h\|_{\cH} }
\]
where the constant $c$ depends on $\w = \mathbf{B}(\omega)$ but is uniform over $\xi$ in the unit sphere and $h$ with $g(X_1(\w+h))\leq \alpha$.

Let $z = X_1(\w+h)$. In that case, it follows from \eqref{eq:WBracket} and \eqref{eq:nilp} that, for all $t \in [0,1]$,

\[ \left \langle  J_{1\to t} [V_i, V_j](X_t),   \xi \right\rangle =\left \langle [V_i, V_j](z),  \xi \right\rangle , \]
and

\[\left \langle J_{1 \to t} V_i (X_t),  \xi  \right\rangle   = \left \langle V_i (z),  \xi \right\rangle +  \sum_j \left \langle [V_i, V_j](z),   \xi \right\rangle  (w+h)^j_{t,1} \]
Letting $\lambda_i = \left \langle V_i (z),  \xi \right\rangle$ and $\mu_{ij} = \left \langle [V_i, V_j](z),   \xi \right\rangle$, we therefore have %(also using Lemma \ref{lem:Hdeltavee})
\[
 \phi_{\xi}(t)  = \left(   \lambda_i + \sum_{j} \mu_{ij}(w+h)^{j}_{t,1} \right)_{i=1,\ldots,d}.
\]
In addition, it follows from \eqref{eq:step2} and compactness of sub-level sets of $g$, that there exists a constant $c_0>0$ s.t.
\[\sum_{i} \lambda_i^2 + \sum_{ij} \mu_{ij}^2 \geq c_0. \]

We then distinguish two cases:\\
\textbf{Case 1}. We first assume that $\sum_{ij} \mu_{ij}^2 \leq \frac{c_0}{2}$, so that $\sup_i \lambda_i \geq \frac{c_0}{2d}$. This case is treated exactly as in the proof of Theorem \ref{thm:elliptic} (even simpler since we do not need rough path estimates). Namely,  choosing $t$ s.t.
\[
\| h\|_{\cH}(1-t)^{\delta-1/2} + \| w \|_{p-var,[t,1]} \leq \frac{c_0}{4d}
\]
in conjunction with Lemma \ref{lem:Hdeltavee} (iii) leads to
\[
\left\|  \phi_{\xi} \right\|_{ (\cH)^{\vee}} \geq \frac{c}{1+\|h\|_{\cH}}.
\]
\textbf{Case 2.} We now assume that $\sum_{ij} \mu_{ij}^2 \geq \frac{c_0}{2}$, so that there exists $i$ for which $\sum_j \mu_{ij}^2 \geq \frac{c_0}{2d}$. But it then follows from Proposition \ref{prop:BH1L2} that there exists a constant $C$ depending only on $\w$ such that
\[
\left\| \lambda_i + \sum_{j} \mu_{ij}(w+h)^{j}_{\cdot,1} \right\|_{H^{1-\delta}} \geq \frac{C (\sum_{j} \mu_{ij}^2)^{1/2}}{1+ \sum_{j} |\mu_{ij}| \| h^j\|_{H^\delta}} \geq \frac{c'}{1+\|h\|_{H^\delta}}
\]
which is enough to conclude using Lemma \ref{lem:Hdeltavee} (ii).
\end{proof}

\begin{remark} \label{rem:SC2} The estimate above could also be applied to obtain convergence for Sussmann and Chitour's continuation method, when the initial control is a Brownian motion. Recall that this method, given a prescribed differentiable path in $\R^n$ $(z(s))_{s\geq 0}$ with $z(0) = X_1(\w)$, consists in obtaining a curve $(h(s))_{s\geq 0}$ in control space s.t. $X_1(\w+h(s)) = z(s)$, by solving the ODE
\[
h(0)=0,\quad\frac{d}{ds} h(s) =  \left[ DX_1(\w+h(s)) \right]^{\dag} (z'(s)),
\]
where $DX_1(\w+h)^{\dag} : \R^n \to \cH$ is the Moore-Penrose pseudo-inverse of the differential $DX_1(\w+h)$. The estimates above imply that for a.e. Brownian (rough) path $\w = \mathbf{B}(\omega)$, 
\[
 \left| DX_1(\w+h)^{\dag}(\xi) \right| \leq C_{\w}(1+\| h \|_{\cH}) |\xi|,
 \]
  and the well-posedness of the ODE then follows from standard arguments (see \cite{Sus93}).\end{remark}

\subsection{Fourier computations}  \label{subsec:Fourier}

We will use the trigonometric basis $(e_k)_{k \geq 0}$ of $L^2 = L^2([0,1],\R)$ defined by
\[
e_0(t) = 1, \;\; \; e_{2m+1}(t) = \sqrt{2} \cos(2\pi m t) , \;\;\; e_{2m+2}(t) = \sqrt{2} \sin(2\pi m t), \; \; m \geq 0.
\]
For $n \geq 0$, we let $I_n = \left\{ k \;:\; 2^{n+1} - 1 \leq k < 2^{n+2}-1 \right\}$ and $\Delta_n$  the orthogonal projection on the span of $\{e_k ; k \in I_n\}$. (Note that $0$ is not in any $I_n$, and all the pairs $2m+1, 2m+2$ belong to the same $I_n$).

We will use that, for some constant $C>0$, for any function $f$ on $[0,1]$, it holds that
\begin{equation} \label{eq:bernstein}
\forall n \geq 0, \quad \frac{1}{C} 2^{-n}\left\| \Delta_n(f)' \right\|_{L^2} \leq \left\| \Delta_n(f) \right\|_{L^2}  \leq C 2^{-n} \left\| \Delta_n(f)' \right\|_{L^2}
\end{equation}
(which is immediate using $(e_{2m+1},e_{2m+2})' = 2\pi m(-e_{2m+2}, e_{2m+1})$). Also note that if $f$ is differentiable, then $\Delta_n(f)' = \Delta_n(f')$.

Finally, we will use the characterization of Sobolev spaces in terms of Fourier coefficients (see e.g. \cite{ST87}) : for any $\delta>0$, there exists $C>0$ s.t., for any $f$ in $L^2([0,1])$, 
\begin{equation} \label{eq:PLBesov}
\frac{1}{C'} \sum_{n \geq 0} 2^{2n \delta} \left\| \Delta_n(f)\right\|_{L^2}^2 \leq \| f \|_{H^{\delta}}^2 \leq C' \sum_{n \geq 0} 2^{2n \delta} \left\| \Delta_n(f)\right\|_{L^2}^2.
\end{equation}

\begin{proof}[Proof of Lemma \ref{lem:Hdeltavee} (ii)]
It is obvious from the definition that $\|\cdot\|_{(H_0^1)^{\vee}} = \|\cdot\|_{L^2}$. 

We now treat the case $\delta \in (1/2,1)$. Note that for a (say smooth) function $h$ on $[0,1]$, it holds that
$h = h'_0 e_0 + \sum_{n \geq 0} \Delta_n(h')$, with $h'_0 = \int_0^1 h $.

In addition,
\begin{align*}
\left\| f \right\|_{(H_0^\delta)^{\vee}} &= \sup_h \frac{ \left\langle f , h' \right\rangle_{L^2} }{\|h\|_{H^\delta}} \geq  \sup \frac{ \sum_{n \geq 0} \left\langle \Delta_n(f) , \Delta_n(h') \right\rangle_{L^2} }{\|h\|_{H^\delta}}\\
\end{align*}
and taking $h$ s.t. $\Delta_n(h') =  2^{n(1-\delta)}\Delta_n(f)$, the result follows from \eqref{eq:PLBesov} and \eqref{eq:bernstein}.
\end{proof}

We then proceed with the proof of Proposition \ref{prop:BH1L2}, for which we will need the following elementary lemma. 
\begin{lemma} \label{lem:DeltaB}
Let $B$ be a Brownian motion, then
\begin{equation*}
\E \left\| \Delta_n (B)' \right\|_{L^2}^2 = 2^{n+1}, \quad Var\left( \left\| \Delta_n (B)' \right\|_{L^2}^2\right) = 2^{n+1},
\end{equation*}
and if $B$, $\underline{B}$ are independent BM's
\[  
\E \left\langle \Delta^n (B)' , \Delta^n( \underline{B})'\right\rangle_{L^2} =0, \quad Var\left(\left\langle \Delta^n (B) , \Delta^n( \underline{B})'\right\rangle_{L^2} \right) = 2^{n+1}.
\]
\end{lemma}

\begin{proof}
This is immediate by noting that, if $B$ is a Brownian motion, it holds that $\gamma_k := B'(e_k) = \int_0^1 e_k(s) dB_s$ , $k \geq 0$, are i.i.d. $\mathcal{N}(0,1)$. Then, we see that
 \[
\left\| \Delta_n (B)' \right\|_{L^2}^2 = \sum_{k=2^{n+1}-1}^{2^{n+2}-2} \gamma_k^2 ,\quad\quad \left\langle \Delta^n (B) , \Delta^n( \underline{B})'\right\rangle_{L^2}  = \sum_{k=2^{n+1}-1}^{2^{n+2}-2} \gamma_k \underline{\gamma}_k
\]
are a sum of $2^{n+1}$ independent r.v.'s with expectation and variance $1$ in the first case and expectation $0$ and variance $1$ in the second case.
\end{proof}

\begin{proof}[Proof of Proposition \ref{prop:BH1L2}]
Define the event
\[ A_n = \left\{  \| \Delta^n (B^i)' \|^2_{L^2} \geq c 2^{n+1} , i=1,\ldots,d,  \quad \sum_{i\neq j}\left|  \left\langle \Delta^n (B^i)' , \Delta^n( {B^j})'\right\rangle_{L^2}   \right| \leq \epsilon 2^{n+1} \right\}.\]
where $0<\epsilon<c < 1$ are fixed.

Note that, by Lemma \ref{lem:DeltaB} and the concentration properties for Gaussian chaoses, it holds that, for some $C>0$, %(TODO check exponent)
\[ \P(A_n ) \geq 1 -  e^{-C 2^{(n+1)/2}} \]
and, in particular, by the Borel-Cantelli lemma, almost surely, $A_n$ holds for $n \geq n_0(\omega)$.

We then write, for $n \geq 0$,
\begin{align*}
\left\|  \sum_{i=1}^d \lambda_i B^i - h  -c \right\|_{H^{1-\delta}}^2& \geq  C' 2^{2(n+1)(1-\delta)} \left\| \Delta_n\left( \sum_{i=1}^d \lambda_i B^i - h \right) \right\|_{L^2}^2 \\
&\geq \frac{C'}{2} 2^{2(n+1)(1-\delta)} \left\| \sum_{i=1}^d \lambda_i  \Delta_n(B^i )\right\|_{L^2}^2 - C'  2^{2(n+1)(1-\delta)} \left\| \Delta_n( h)  \right\|_{L^2}^2 \\
&\geq 2^{-2(n+1) \delta} \frac{C'}{2C} \left\| \sum_{i=1}^d \lambda_i  \Delta_n(B^i)' \right\|_{L^2}^2 - (C')^2 C 2^{2(n+1)(1-2 \delta)} \left\| h  \right\|_{H^{\delta}}^2
\end{align*}
where we have used \eqref{eq:PLBesov} and \eqref{eq:bernstein}.

Then assuming that $A_n$ holds, the first term is bounded from below by $\frac{C'(d c - \epsilon)}{2C}2^{-(n+1)(2\delta -1)}$, and in particular, this yields,
\[
\| h \|_{H^\delta}^2 \leq{ \frac{(d c - \epsilon)}{4C' C^2} }2^{-(n+1)(2\delta -1)} \Rightarrow \left\|  \sum_{i=1}^d \lambda_i B^i - h  -c \right\|_{H^{1-\delta}}^2\geq  { \frac{C'(d c - \epsilon)}{4C}} 2^{-(n+1)(2 \delta -1)}.
\]
Recalling that a.s. $A_n$ holds for $n$ large enough, this clearly implies \eqref{eq:BH1L2} for a suitable (random) $C$.

\end{proof}

\section{Continuity} \label{sec:cont}

In this section, we aim at showing that if, for a given initial condition $\mathbf{w}$, the gradient flow converges to a non-degenerate minimiser, this will also be the case for initial conditions close to $\mathbf{w}$ in rough path metric. In addition, we also want to consider the case where $\cH$ is replaced by some (for instance finite dimensional) subspace $\tilde{\cH}$. With this in mind, for $\tilde{\cH} \subset \cH \subset C^{q-var}$, we define
\begin{equation} \label{eq:defEps}
 \epsilon_{q}(\tilde{\cH},\cH) = \sup_{h \in \cH} \frac{ \| h - \tilde{\pi}(h)\|_{q-var;[0,1]}}{\|h\|_{\cH}}
\end{equation}
where $\tilde{\pi}$ is the orthogonal projection ${\cH} \to \tilde{\cH}$.

The main result is then the following.

\begin{proposition} \label{prop:continuityGF}
Let $\ww$ $\in$ $C^{p-var}_g$, $\cH \subset C^{q-var}$ for $1 \leq q \leq p$ with $\frac{1}{p} + \frac{1}{q} > 1$, $V_i$ be $C^{\infty}_b$ vector fields and $g \in C^2(\R^n)$ satisfy \eqref{eq:locLoj}. Assume that the gradient flow defined by \eqref{eq:roughgf} satisfies
\[ \lim_{s \to \infty} h(s) = h_{\infty}\mbox{ in }\cH, \mbox{ where } \cL_{\ww}(h_{\infty}) = 0 \mbox{ and } c_{\ww}(h_{\infty}) > 0. \]
Then, for any $\delta > 0$, there exists $\varepsilon > 0$ such that, if $\tilde{\ww} \in C^{p-var}$, $\tilde{\cH} \subset \cH$ are such that
\[ d_{\alpha-r.p.} (\ww, \tilde{\ww}) \leq \varepsilon, \quad \epsilon_q(\tilde{\cH},\cH) \leq \varepsilon,\]
then, letting $\tilde{h}$ the gradient flow on $\tilde{\cH}$ associated to $\tilde{\ww}$, $\tilde{\cH}$, it holds that,
\[ \exists \tilde{h}_{\infty} \in \tilde{\cH}, \quad \lim_{s \to \infty} \tilde{h}(s) = \tilde{h}_{\infty} \mbox{ in }\tilde{\cH}, \mbox{ and }  \cL_{\tilde{\ww}}(\tilde{h}_{\infty}) = 0 , \; \| h_{\infty} - \tilde{h}_{\infty} \|_{C^{q-var}} \leq \delta.\]
\end{proposition}

\begin{proof}
We first estimate the difference between $h$ and $\tilde{h}$ on a finite time interval $[0,S]$. Let
\[ R :=  \max \left(\left\| \ww \right\|_{p-var} , \left\| \tilde{\ww} \right\|_{p-var} , \sup_{s \geq 0} \left\| h(s)\right\|_{\cH} \right)\]
which is finite by assumption. We then write, for any $s \geq 0$,
\modif{\begin{align*}
h(s) - \tilde{h}(s) = &-\int_0^s \nabla_{\cH} \cL_{\ww} (h(u)) du +  \int_0^s \nabla_{\tilde{\cH}} \cL_{\tilde{\ww}} (\tilde{h}(u)) du \\
=& \int_0^s \left(\nabla_{\cH} \cL_{\tilde{\ww}} (h(u)) - \nabla_{\cH} \cL_{\ww} (h(u))  \right) du \\
& + \; \int_0^s \left( \nabla_{\cH} \cL_{\tilde{\ww}} (\tilde{h}(u)) - \nabla_{\cH} \cL_{\tilde{\ww}} (h(u))  \right) du \\
&+ \;   \int_0^s \left( \nabla_{\tilde{\cH}} \cL_{\tilde{\ww}} (\tilde{h}(u)) - \nabla_{\cH} \cL_{\tilde{\ww}} (\tilde{h}(u))  \right) du . 
\end{align*}}
We then bound these three terms in $q$-variation, using the continuity properties of $\nabla_\cH \cL$ w.r.t. its arguments to obtain that
\[ 
\left\|h(s) - \tilde{h}(s) \right\|_{q-var} \leq C_1 s d_{p-r.p.} (\ww,\tilde{\ww}) + C_1 \int_0^s \left\| h(u) - \tilde{h}(u) \right\|_{q-var} du + C_1 s  \epsilon_q(\tilde{\cH},\cH),
\]
assuming we have an a priori bound on $\tilde{h}$, for instance
\[ \sup_{0 \leq u \leq s} \left\| \tilde{h}(u) \right\|_{q-var} \leq R+1, \]
\modif{where the constant $C_1$ depends on $R$.}  By Gronwall's lemma, this yields
\begin{equation} \label{eq:estDiffS}
\left\|h(S) - \tilde{h}(S) \right\|_{q-var}  \leq C_1 S e^{C_1 S} \left( d_{p-r.p.} (\ww,\tilde{\ww}) +  \epsilon_q(\tilde{\cH},\cH) \right)
\end{equation}
for every $S>0$ (as long as the r.h.s. is smaller than $1$).

\modif{Our goal for the following is to deduce the convergence from Corollary \ref{cor:LocLoj}.} By assumption on $h_{\infty}$, and continuity of $c$ w.r.t. its parameters, we can find $0 < r \leq \modif{\frac{\delta}{2}} $ \modif{($r<1$)} s.t.
\[
c^{2r}_{\ww}(h_{\infty}) := \inf \left\{ c_{\tilde{\ww}}(h), \;\;  d_{p-r.p.}(\ww,\tilde{\ww}) + \| h - h_{\infty} \|_{q-var} \leq 2 r \right\} > 0.
\]
\modif{
Let $K \subset \R^n$ be a compact set containing $\{X_1^x(\tilde{\w}+h), \, d_{p-r.p.} (\ww,\tilde{\ww}) + \|h\|_{\tilde{\cH}} \leq 1\}$, and $c_{g,K}$ the \L{}ojasiewicz constant for $g$ on $K$. We now fix $S$ such that
\begin{equation*}
\left\|h(S) -  h_{\infty} \right\|_{q-var}  \leq \frac{r}{2}.
\end{equation*}
Additionally, from the continuity of $\cL$, we can assume that
\begin{equation}\label{eq:bound-in-r-for-L}
\cL_{\ww}(h(S))\leq\frac{1}{2}\left(\frac{r}{4}c_{g,K}c^{2r}_{\ww}(h_{\infty})\right)^2,
\end{equation}}
\modif{where the quantity on the r.h.s. is chosen in such a way to make Corollary \ref{cor:LocLoj} applicable later.}

For arbitrary $\w \in C^{p-var}_g, h \in C^{q-var}$ we can define $\tilde{c}_{\w}(h)$ analogously to \eqref{eq:defc}-\eqref{eq:FormulaGradient} by
\[
\tilde{c}_{\w}(h)^2 := \inf_{|\xi|_{\R^n}=1}   \left\| \xi \cdot_{\R^n} \nabla_{\tilde{\cH}} X_1 \right\|_{\cH}^2 = \inf_{|\xi|_{\R^n}=1}  \left\|   t \mapsto \left( \langle J_{1 \leftarrow t} V_i(X_t)  \cdot \xi \right)_{i=1,\ldots,d} \right\|_{\tilde{\cH}^{\vee}}^2.
\]
Note that for any path $f \in C^{p-var}([0,1],\R^d)$, it holds that
\begin{align*}
\|f \|_{\cH^\vee} &= \sup_{\|h\|_{\cH} \leq 1} \int_0^1 f(t) dh(t) \\
& \leq   \sup_{\|h\|_{\tilde{\cH}} \leq 1} \int_0^1 f(t) dh(t) + \sup_{\|h\|_{\cH} \leq 1} \int_0^1 f(t) d \left(h(t) - \tilde{\pi}(h)(t) \right) \\
&\leq \|f \|_{\tilde{\cH}^\vee}  + C_{p,q} \| f \|_{p-var}  \epsilon_q(\tilde{\cH},\cH),
\end{align*}
by Young integration.

It follows that for any $\ww$, $h$ with   $\|\ww\|_{p}$ and $\|h\|_{q-var}$ smaller than $R+1$, it holds that
\begin{equation} \label{eq:tildeC}
 c_{\ww}(h) \leq   \tilde{c}_{\ww}(h) + C_2   \epsilon_q(\tilde{\cH},\cH),
\end{equation}
where $C_2>0$ depends on $R$.

We then fix $\varepsilon > 0$ s.t.
\begin{equation*}
2 C_1 S e^{C_1 S} \varepsilon \leq \frac{r}{2},\quad
\end{equation*}
and note that we then have, by \eqref{eq:estDiffS},
\begin{equation*}
\left\|h(S) -  \tilde{h}(S) \right\|_{q-var}  \leq \frac{r}{2}.
\end{equation*}

We also assume that 
\[ C_2 \varepsilon \leq \frac{1}{2} c^{2r}_{\ww}(h_{\infty}), \quad \varepsilon \leq \modif{\frac{r}{2}}, \]
from which we deduce by \eqref{eq:tildeC} that
\[ \tilde{c}^r_{\tilde{\ww}}(\tilde{h}(S)) := \inf_{\| h - \tilde{h}(S) \| \leq r} \tilde{c}_{\tilde{\ww}}(h) \geq  \frac{1}{2} c^{2r}_{\ww}(h_{\infty}).\]

Assuming additionally that $\varepsilon>0$ is s.t.
\begin{equation*}
d_{p-r.p.} (\ww,\tilde{\ww}) ,  \epsilon_q(\tilde{\cH},\cH) \leq \varepsilon \; \Rightarrow \; 2  \sqrt{\cL_{\tilde{\ww}}(\tilde{h}(S)) } \leq  \frac{r}{2} c_{g,K} c^{2r}_{\ww}(h_{\infty})
\end{equation*}
(which is always true for $\varepsilon$ small enough by continuity of $\cL$ \modif{and \eqref{eq:bound-in-r-for-L}}), we can combine the above inequalities with Proposition \ref{prop:LojL} and the convergence criterion in Corollary \ref{cor:LocLoj} to obtain that
\[
\lim_{s \to \infty} \tilde{h}(s) = \tilde{h}_{\infty}, \mbox{ with } \cL_{\tilde{\ww}}(\tilde{h}_{\infty}) = 0 \mbox{ and } \| \tilde{h}_{\infty} - \tilde{h}(S) \|_{\tilde{\cH}} \leq r.
\]
It only remains to obtain the bound $\|h_{\infty} - \tilde{h}_{\infty} \|_{q-var}$, which follows immediately from the triangle inequality.
\end{proof}

In order to apply the above result, we need assumptions guaranteeing that $\epsilon_q(\tilde{\cH},\cH)$ defined in \eqref{eq:defEps} is small. We first have a general compactness criterion guaranteeing convergence to $0$.

\begin{lemma}
Let $\cH$ be compactly embedded in $C^{q-var}$, and let $\cH_n$, $n\geq 0$, be a sequence of subspaces of $\cH$, with $\cH_n \subset \cH_{n+1}$ and $\cup_n {\cH}_n = \cH$. Then 
\[ \lim_{n \to \infty} \epsilon_q\left(\cH_n, \cH\right) = 0.\]
\end{lemma}

\begin{proof}
Assume by contradiction, that $\limsup_n \epsilon_q\left(\cH_n, \cH\right)  > 0$, and by taking a subsequence, we can in fact assume that $\lim_n \epsilon_q\left(\cH_n, \cH\right) = \epsilon > 0$. We can then find a sequence $h_n$ such that 
\[
\|h_n\|_{\cH} \leq 1,  \;\; h_n \in \cH_n^{\perp}, \;\; \liminf_n \| h_n \|_{q-var} \geq \varepsilon > 0.
\]
Taking again a subsequence if necessary, by the compact embedding of $\cH$ in $C^{q-var}$, there exists $h \in \cH$ s.t. $h_n \to h$ in $C^{q-var}$. On the other hand, the fact that $h_n \in \cH_n^{\perp}$, along with the assumption on $\cH_n$, implies that $h_n$ converges weakly to $0$ in $\cH$. This implies that $h=0$, i.e. $\| h_n\|_{q-var} \to 0$, which is a contradiction. 
\end{proof}

We can also compute directly upper bounds in many cases of interests. For instance, we have the following on piecewise linear approximations when $\cH=H^1$ (extensions to other Besov spaces would be similar).

\begin{lemma} \label{lem:epsH}
Let $\cH = H^1_0$, and let $\cH_n$ be the subspace of $\cH$ consisting of functions which are piecewise linear on each interval of the form $[i/n, (i+1)/n]$, $i=0, \ldots, n-1$. Then, for each $1 < q < 2$, it holds that
\[
\epsilon_q\left(\cH_n \right) \leq C_q n^{\frac{1}{q}-1}.
\]
\end{lemma}

\begin{proof}
We fix $n$ and note that the orthogonal projection $\pi_n(f)$ of $f \in  \cH$ is simply the piecewise linear approximation of $f$ which coincides with $f$ on each grid-point $t_i = \frac{i}{n}$. It follows that $\cH_n^\perp = \{ f \in H^1_0, f_{t_i} = 0, i = 1,\ldots, n\}$. We now fix $h$ in $\cH_n^\perp$ . For $(s_j)$ an arbitrary partition of $[0,1]$,  note that if $s_j \leq t_i \leq t_{i+k} \leq s_{j+1}$, then
\[ 
\left| h_{s_j,s_{j+1}} \right|^q  = \left| h_{s_j,t_i}  + h_{t_{i+k}, s_{j+1}}\right|^q \leq C_q \left( \left| h_{s_j,t_i} \right|^q   + \left| h_{t_{i+k}, s_{j+1}}\right|^q \right)
\]
Letting $\cP$ be the set of all partitions of $[0,1]$, and $\cP_{t}$ be the subset of those which contain $\{t_i, i=1,\ldots, n\}$, this yields
\begin{align*}
  \| h\|_{q-var;[0,1]}^q =  \sup_{(s_j) \in \cP} \sum_j \left| h_{s_j,s_{j+1}} \right|^q  &\leq  C_q \sup_{(s_j) \in \cP_t}  \left| h_{s_j,s_{j+1}} \right|^q = C_q \sum_{i=0}^{n-1}  \| h\|_{q-var;[t_i,t_{i+1}]}^q.
\end{align*}

On the other hand, for any $t_i\leq t_{i+1}$, it holds by Cauchy-Schwarz, that
\[
 \| h\|_{q-var;[t_i,t_{i+1}]} \leq  \| h\|_{1-var;[t_i,t_{i+1}]}  \leq n^{-1/2} \left( \int_{t_i}^{t_{i+1}} |\dot{h}|^2 \right)^{1/2} .
\]

We therefore obtain (using H\"older's inequality in the sum below) that, for every $h \in \cH_n^\perp$,
\begin{align*}
\| h \|_{q-var}^q  &\leq C n^{-q/2} \sum_i \left( \int_{t_i}^{t_{i+1}} |\dot{h}|^2 \right)^{q/2} \\
& \leq C n^{- q/2} n^{1-q/2}  \left( \sum_i\int_{t_i}^{t_{i+1}} |\dot{h}|^2 \right)^{q/2} \\
& = C n^{1-q} \| h \|_{H^1}^q.
\end{align*}
By definition of $\epsilon_q$, this proves the claim. %(TODO Double check)
\end{proof}

In particular, we can then use our continuity results to obtain convergence for gradient flow on piecewise linear controls, if we know convergence for a continuous problem, such as in the step-$2$ nilpotent case.

\begin{corollary}
For $n \geq 1$, let $\cH_n \subset H^1_0$ be the subset consisting of functions which are piecewise linear on each interval of the form $[i/n, (i+1)/n]$, $i=0, \ldots, n-1$. Let $h^{n,0}$ be a $\cH_n$-valued random variable s.t. the $\dot{h}^j_{[i/n, (i+1)/n]}$ are i.i.d. centered Gaussians with variance $\frac{1}{n}$.

Assume that $g\in C^2(\R^n)$ satisfies \eqref{eq:locLoj}-\eqref{eq:CompSub}, and that the $V_i$ are $C^{\infty}_b$ and satisfy \eqref{eq:nilp}-\eqref{eq:step2}, and let $\cL : \cH_n \to \R $ defined by $\cL(h) = g(X_1^x(h))$. Let $(h^{n}(s))_{s \geq 0}$ be defined by
\[
h^n(0)= h^{n,0}, \quad \frac{d}{ds} h^n(s) = - \nabla_{\cH_n} \cL(h^n(s)).
\]
Then it holds that
\[
\lim_{n \to \infty} \P\left( \exists h_{\infty} \in \cH_n, \; \lim_{s \to \infty} h^n(s) = h_{\infty}, \;\;\cL(h_{\infty})= 0 \right) = 1.
\]
\end{corollary}

\begin{proof}
Let $B$ be a Brownian motion and $B^n$ its piecewise linear (on intervals $[i/n,(i+1)/n]$) approximation. For fixed $p>2$, letting further $\mathbf{B}$ be the rough path lift of $B$ and $\mathbf{B}^n$ that of $B^n$, it holds from classical rough path results (e.g. \cite[Corollary 13.21]{FV10}) that,
\begin{equation} \label{eq:pwrp}
 \lim_{ n \to \infty} \rho_{p-var}\left(\mathbf{B}, \mathbf{B}^n\right) = 0 \mbox{  in probability}.
\end{equation}
Let $v^n$ be the solution to the gradient flow equation 
\[
v(0)=0, \quad \dot{v}(s) =  - \nabla_{\cH_n} \cL_{\mathbf{B}^n} (v(s)).
\]
By Theorem \ref{thm:mainstep2}, Proposition \ref{prop:continuityGF} and Lemma \ref{lem:epsH}, \eqref{eq:pwrp} implies that
\[
\lim_{n \to \infty} \P \left( \lim_{s \to \infty} v^n(s) = v^n_{\infty} \in \cH_n \mbox{ with } \cL_{\mathbf{B}^n}(v_{\infty}^n)= 0 \right)=1.
\]
On the other hand, since $B^n$ is regular (of bounded variation), the rough path sum $\mathbf{B}^n + v^n$ is simply (the rough path lift of) $B^n + v^n$. In particular, $h^n = v^n + B^n$ satisfies
\[
h^n(0)=B^n, \quad \dot{h}^n(s) = - \nabla_{\cH_n} \cL(h^n(s)).
\]
and convergence of $v^n(s)$ as $s \to \infty$ is equivalent to that of $h^n(s)$. 

Since  $h^{n,0}$ in the statement of the corollary has the same law as $B^n$, this proves the claim.
\end{proof}

As a simple consequence of the continuity results, we can also show that, in a general hypoelliptic situation, the gradient flow will converge with positive probability when started from Brownian motion.

\begin{proposition}
Let $\ww = \mathbf{B}(\omega)$ be enhanced Brownian motion. Assume that $\cH \supset C^{\infty}$, the $V_i$ are bracket-generating, and $g$  $\in C^2(\R^n)$ satisfies \eqref{eq:locLoj}.
Then, for the corresponding (random) gradient flow defined in section \ref{sec:defGF}, it holds that
\[
\P \left( \lim_{s \to \infty} h(s) = h_{\infty} \in \cH \mbox{ with } \cL_{\mathbf{B}(\omega)} (h_{\infty}) = 0 \right) > 0.
\]
\end{proposition}

\begin{proof}
It is classical that there exists a control $h$ in $C^{\infty}$ such that $X_1(h) = y$ and $\nabla_{\cH} X_1 (h)$ is non-degenerate, i.e., $c(h) \neq 0$ (see \cite[Proposition 1.4.6]{Rif14}).

(Note that usually the non-degeneracy is stated for $\cH = H^1$, but it is in fact true for the gradient in any Hilbert space $\cH$ rich enough that $\cH^{\vee}$ is a norm (this is true if $\cH \subset C^{\infty}$). Indeed, by \eqref{eq:FormulaGradient}, it is then equivalent to the fact that for any non-zero $\xi$, $t \in [0,1] \mapsto \left( \langle J_{1 \leftarrow t} V^i(X_t)  \cdot \xi \right)_{i=1,\ldots,d}$ does not vanish identically, which does not depend on the choice of $\cH$.)

Fix $p > 2$ and let $\mathbf{h}$ be the lift of $h$ to $C^{p-var}_g$. By Proposition \ref{prop:continuityGF}, there exists $\delta >0$ such that 
\[ \rho_{p-r.p.}(\ww, \mathbf{h}) \leq \delta \; \Rightarrow \; \mbox{ the gradient flow started at $\ww$ converges to a zero of $\cL_{\ww}$.}
\]
On the other hand, since the support of enhanced Brownian motion contains the lifts of all smooth paths (cf. \cite{FV10}), the left hand side has positive probability for $\ww = \mathbf{B}(\omega)$.
\end{proof}

\begin{remark}
The only property of enhanced Brownian motion that we use is that its law has full support on rough path space, so that the result also holds for any random rough path with this property (for instance, fractional Brownian motion with Hurst parameter $H > \frac{1}{4}$, see e.g. \cite{FV10}).
\end{remark}

\section{Numerical experiments} \label{sec:num}
We design a numerical experiment with the following setup. We choose $d$ vector fields $V_1,\dots,V_d$ over $\R^n$. We draw two points $x,y\in\R^n$, randomly sampled from the normal distribution $\mathcal{N}(0,I)$ in $\R^n$, and a trajectory $w^H:[0,1]\rightarrow\R^d$ of a fractional Brownian motion with Hurst parameter $H\in(1/4,1)$. We then introduce the functional $\cL:\mathcal{H}\rightarrow\R$ such that,
for every $h\in\cH$,
\[
\cL(h) = \frac{1}{2}\left|X_1^x(w^H+h)-y\right|^2, 
\]
where $X^x(w^H+h)$ is the solution to the (rough) differential equation
\begin{equation}\label{eq:rde-h-num}
dX_t = \sum_i V_i (X_s) d (w^H + h)^i_s, \;\; X_0=x.
\end{equation}
Our goal is minimising the functional $\cL$ by means of a gradient flow defined on $\cH$. Specifically, we consider the solution $\left(h(s)\right)_{s\geq 0}$ to the $\cH$-valued ODE
\begin{equation}\label{eq:ode-gf-h}
  h(0)=0,\quad\frac{d}{ds}h(s)=-\nabla_{\cH} \mathcal{L}(h(s))  
\end{equation}
and study the behaviour of $\cL(h(s))$ for increasing values of $s\geq 0$. Notice that the coupled system \eqref{eq:rde-h-num}-\eqref{eq:ode-gf-h} has a forward-backward structure, in the sense that, in order to compute the value of $\cL(h(s))$ at time $s\geq 0$, one needs to compute the solution of the forward-in-time RDE \eqref{eq:rde-h-num}. Meanwhile, the computation of $\nabla_{\cH} \mathcal{L}(h(s))$, as suggested by \eqref{eq:nablaL} and \eqref{eq:formula-eq-gradient}, would require the computation of a backward-in-time RDE of the form \eqref{eq:formula-eq-jacobian}.

We discretise the dynamics \eqref{eq:rde-h-num} and the gradient flow \eqref{eq:ode-gf-h}. In detail, let $L\in\mathbb{N}$ be the number of time steps for the dynamics \eqref{eq:rde-h-num} with step size $\Delta t = 1/L$, and define, for $l=0,\dots,L$, the time steps $t_l=l\,\Delta t\in[0,1]$. For a continuous path $w:[0,1]\rightarrow\R^d$, denote its increments $w_{t_l,t_{l+1}}=w_{t_{l+1}}-w_{t_l}$ for every $l=0,\dots,L-1$. Let $K\in\mathbb{N}$ be the number of time steps for the gradient flow \eqref{eq:ode-gf-h}, or gradient descent updates, with step size $\Delta s$. Then, initialising the increments $h_{t_l,t_{l+1}}^0=0\in\R^d$ at step $k=0$, at any step $0\leq k\leq K$, given the increments $\{h_{t_l,t_{l+1}}^{k}\}_{l=0}^{L-1}$, $X^k_1$ is computed as the final value of the recursion
\begin{align}\label{eq:discrete-dyn}
    	\begin{split}
	&X_{t_0}^k = x,\\ 
    	&X_{t_{l+1}}^k = X_{t_l}^k + \sum_{i=1}^d V_i\left(X_{t_l}^k\right)\left(w_{t_l,t_{l+1}}^{H,i}+h_{t_l,t_{l+1}}^{k,i}\right)
	+\frac{1}{2}\sum_{i=1}^d DV_i\left(X_{t_l}^k\right)V_i\left(X_{t_l}^k\right)\Delta t^{2H},\\	&\quad\quad\quad\quad\quad\quad\quad\quad\quad\quad\quad\quad\quad\quad\quad\quad\quad\quad\quad\quad\quad\quad\quad\quad\quad\quad\quad\quad\quad\quad \forall l=0,\dots,L-1.
	\end{split}
\end{align}  

Note that the classical discretization scheme for rough differential equations would require iterated integrals of order $2$ (for $ H > \frac{1}{3}$ ) and $3$ (for $H > \frac{1}{4}$), which are difficult to simulate. In the scheme above we simply replace these iterated integrals by their expectations (only the term of order $2$ is non zero). Convergence of this scheme (along with rate) has been proven (in the driftless case) by \cite{LT19} for $H>\frac{1}{3}$ and \cite{Hua23} for $H > \frac{1}{4}$.

\modif{Moreover, note that the use of the discretization scheme for the computation of \eqref{eq:rde-h-num} induces a numerical scheme for the computation of $\nabla_{\cH}\cL$ in \eqref{eq:ode-gf-h}. In detail,} define the finite-dimensional functional $(\R^d)^L\ni \{h_{t_l,t_{l+1}}\}_{l=0}^{L-1}\mapsto \cL^L\left(\{h_{t_l,t_{l+1}}\}_{l=0}^{L-1}\right)=\frac{1}{2}|X_1-y|^2$ where $X_1$ is the final value of the recursion \eqref{eq:discrete-dyn} with $\{h_{t_l,t_{l+1}}\}_{l=0}^{L-1}\in(\R^d)^L$ as increments. \modif{Interpreting $\cH^L:=(\R^d)^L$ as a finite dimensional subspace of $\cH$, the increments of $\nabla_{\cH^L} \mathcal{L}$ are given by
\[
\nabla_{\cH^L} \mathcal{L}\left(h^k\right)_{t_l,t_{l+1}} := \Delta t \, \frac{\partial \cL^L}{\partial h_{t_l,t_{l+1}}}\left(h_{t_l,t_{l+1}}^{k}\right),\;\;\forall l=0,\dots,L-1.
\]}
The $(k+1)$-th gradient descent update of the increments then reads as
\begin{equation}\label{eq:gradient-descent}
    h_{t_l,t_{l+1}}^{k+1} = h_{t_l,t_{l+1}}^{k} - \Delta s\, \Delta t \, \frac{\partial \cL^L}{\partial h_{t_l,t_{l+1}}}\left(h_{t_l,t_{l+1}}^{k}\right),\;\;\forall l=0,\dots,L-1.
\end{equation}
Above, the quantity $\alpha_L = \Delta s \Delta t = \Delta s / L$ can be interpreted as the learning rate of the gradient descent for the finite-dimentional minimisation problem associated with $\cL^L$.

\modif{The above construction provides a consistent finite-dimensional approximation of the gradient flow in $\cH$. The error analysis of $\nabla_{\cH^L} \mathcal{L}$ with respect to $\nabla_{\cH} \mathcal{L}$ is a direction for future research. }

\begin{figure}[ht]
    \centering
    \begin{subfigure}[b]{0.49\textwidth}
        \centering
        \includegraphics[width=\textwidth]{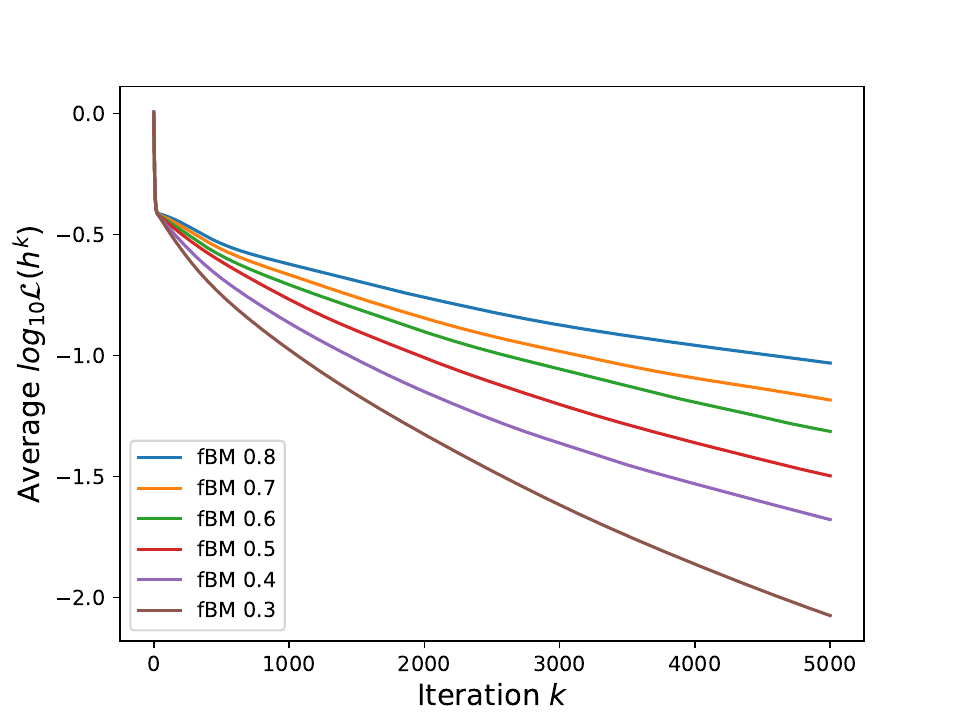}
        \caption{Example \ref{ex:step2nilpotent} with $n=55$ and $d=10$.}
        \label{fig:step2}
    \end{subfigure}
    \begin{subfigure}[b]{0.49\textwidth}
        \centering
        \includegraphics[width=\textwidth]{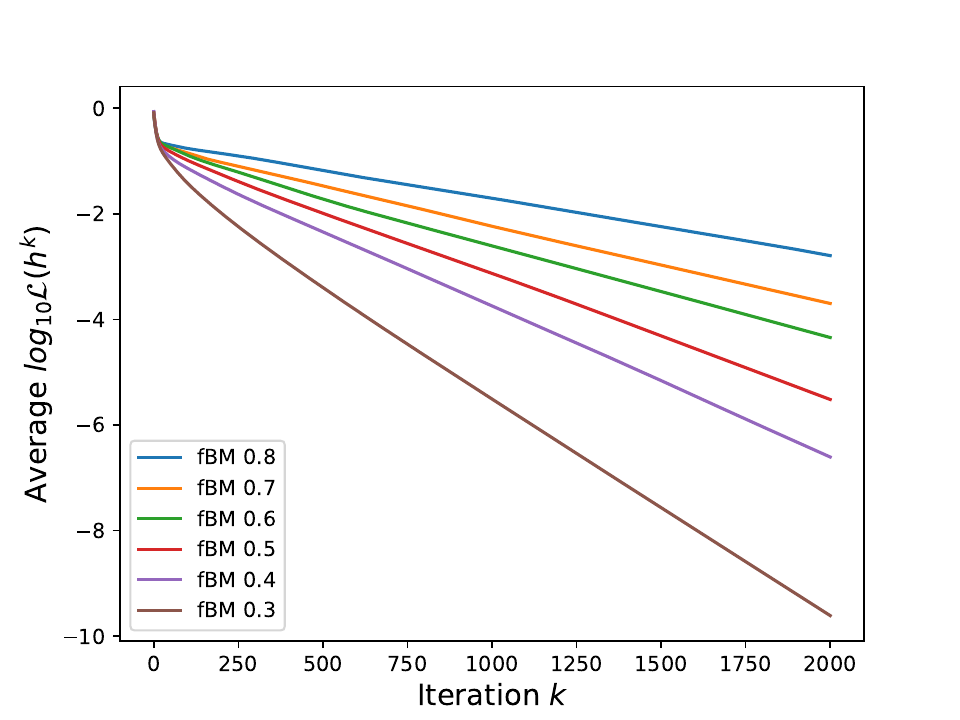}
        \caption{Example \ref{ex:step3nilpotent} with $n=5$ and $d=2$.}
        \label{fig:step3}
    \end{subfigure}
    \caption{Averaged evolution of $\log_{10} \mathcal{L}\left(h^k\right)$ as a function of the gradient descent iterations $k$ for different values of the Hurst parameter $H$.}
    \label{fig:numerics}
\end{figure}

We introduce two examples for numerical simulation.
\begin{example}[Step-2 nilpotent, Example 16.45 in \cite{FV10}]\label{ex:step2nilpotent}
For $d\geq 1$ and $n=d(d+1)/2$, we can identify $\R^n$ with $\R^d \times so(d)$ where $so(d)$ is the space of anti-symmetric $d\times d$ matrices, identified with its upper-diagonal elements. For $i=1,\dots,d$, we consider the $\R^d \times so(d)$-valued vector field $V_i$ such that
\[
V_i(x,x') = \partial_i + \frac{1}{2}\left( \sum_{1\leq j < i } x_j\partial_{j,i} - \sum_{ i < j \leq d} x_j\partial_{i,j} \right),\quad \forall (x,x')\in \R^d \times so(d),
\]
where $\partial_i$ denotes the $i$-th coordinate vector field on $\R^d$ and $\{\partial_{m,j}:1\leq m<j\leq d\}$ are the coordinate vector fields on $so(d)$. Notice that the family of vector fields $\{V_i:i=1,\dots,d\}$ is bracket generating and step-2 nilpotent in the sense of \eqref{eq:nilp}. 
\end{example}
\begin{example}[Step-3 nilpotent]\label{ex:step3nilpotent}
For $d=2$ and $n=5$, we consider the $\R^5$-valued vector fields $V_1$ and $V_2$ such that, for every $x\in\R^5$,
\[
\begin{split}
V_1(x) &= \partial_1-\frac{1}{2}x_2\partial_3- \left(\frac{1}{12}x_1x_2+ \frac{1}{2} x_3\right)\partial_4 -\frac{1}{12} x^2_2\partial_5,\\
 V_2(x) &= \partial_2+ \frac{1}{2}x_1\partial_3+\frac{1}{12}x^2_1\partial_4+\left( \frac{1}{12}x_1x_2- \frac{1}{2} x_3\right)\partial_5,
\end{split} 
\]
where $\{\partial_i:i=1,\dots,5\}$ are the coordinate vector fields on $\R^5$. We then define and compute
\[
V_3:=[V_1,V_2]=\partial_3+\frac{1}{2}x_1 \partial_4+ \frac{1}{2}x_2\partial_5,\quad V_4:=[V_1,V_3]=\partial_4,\quad V_5:=[V_2,V_3]=\partial_5.
\]
Notice that the vector fields $V_1$ and $V_2$ are bracket generating and it holds that, 
\[
\forall x\in\R^5,\forall i\in\{1,\dots,5\},\forall j\in\{4,5\},\quad [V_i,V_j](x)=0,
\]
which, in this case, is equivalent to a step-3 nilpotent condition.
\end{example}

We consider a multi-dimensional Example \ref{ex:step2nilpotent} in $\R^{55}$ with $d_1=10$ vector fields, and a lower-dimensional Example \ref{ex:step3nilpotent} in $\R^5$ with $d_2=2$ vector fields. In the numerical simulations, we set $L=100$, $\Delta s=0.1$, and the number of gradient descent updates, $K_1=5000$ and $K_2=2000$, respectively.   
For each example and each Hurst parameter $H\in\{0.3,0.4,\dots,0.8\}$, we repeat the iterated procedure \eqref{eq:gradient-descent} 100 times with independent points and initial trajectory sampling. We compute the average order of loss $\log_{10} \mathcal{L}\left(h^k\right)$ at each step $k$ and plot its evolution in Figure \ref{fig:numerics}. We observe that, for both examples, the value of $\mathcal{L}\left(h^k\right)$ decays exponentially with $k$. While Theorem \ref{thm:mainstep2} explains this fact in the case of Example \ref{ex:step2nilpotent} (at least for $H\leq 0.5$), the assumptions \eqref{eq:nilp}-\eqref{eq:step2} do not cover Example \ref{ex:step3nilpotent}. This suggests that the convergence result may also hold under more general assumptions on the vector fields. Determining whether the step-$N$ nilpotent condition or the more general H\"ormander condition is sufficient to prove convergence is a natural and interesting direction for future research. Furthermore, we observe from Figure \ref{fig:numerics} that the speed of convergence depends on the value $H$ of the Hurst parameter characterising the initial trajectory. More precisely, the lower the $H$, the faster the asymptotic decay of $\mathcal{L}\left(h^k\right)$. In particular, there is a non-negligible performance gain when sampling, for example, trajectories of fraction Brownian motion with Hurst parameter $H=0.3$ compared to trajectories of Brownian motion corresponding to $H=0.5$. This phenomenon is not fully explained by our theoretical findings, making its understanding an interesting avenue for future research, particularly in view of applications. 
\bibliographystyle{alpha}
\bibliography{rough_control-v2}

\end{document}